\documentclass[11pt, twoside, leqno]{article}

\usepackage{amssymb}
\usepackage{amsmath}
\usepackage{amsthm}
\usepackage{color}
\usepackage{mathrsfs}

\allowdisplaybreaks	

\def\rr{{\mathbb R}}
\def\rn{{{\rr}^n}}
\def\zz{{\mathbb Z}}
\def\cc{{\mathbb C}}
\def\nn{{\mathbb N}}

\def\cs{{\mathcal S}}

\def\fz{\infty}
\def\az{\alpha}
\def\bz{\beta}

\def\gz{{\gamma}}
\def\bgz{{\Gamma}}

\def\lz{\lambda}
\def\blz{\Lambda}
\def\oz{{\omega}}

\def\tz{\theta}

\def\lf{\left}
\def\r{\right}

\def\hs{\hspace{0.25cm}}
\def\ls{\lesssim}

\def\ov{\overline}
\def\noz{\nonumber}
\def\wz{\widetilde}
\def\wh{\widehat}
\def\st{\subset}
\def\com{\complement}
\def\bh{\backslash}

\def\dist{\mathop\mathrm{\,dist\,}}
\def\supp{\mathop\mathrm{\,supp\,}}

\def\loc{\mathop\mathrm{\,loc\,}}
\def\div{\mathop\mathrm{div}}
\def\essinf{\mathop\mathrm{\,ess\,inf\,}}

\def\dint{\displaystyle\int}

\def\rnn{\rr_+^{n+1}}

\def\nab{\nabla}

\newtheorem{thm}{Theorem}[section]
\newtheorem{prop}[thm]{Proposition}
\newtheorem{lem}[thm]{Lemma}

\theoremstyle{definition}
\newtheorem{defn}[thm]{Definition}
\newtheorem{rem}[thm]{Remark}

\renewcommand{\vec}[1]{\boldsymbol{#1}}
\numberwithin{equation}{section}

\begin{document}

\arraycolsep=1pt

\title{\bf\Large Weighted $L^p$ Estimates of Kato Square Roots
Associated to Degenerate Elliptic Operators
\footnotetext{\hspace{-0.35cm} 2010 {\it
Mathematics Subject Classification}. Primary 47B06; Secondary 46E30, 35J70, 42B30, 42B35.
\endgraf {\it Key words and phrases}. Kato square root,
degenerate elliptic operator, Riesz transform,
Lebesgue space, Hardy space, square function, Muckenhoupt weight.
\endgraf Dachun Yang is supported by the National
Natural Science Foundation of China (Grant Nos. 11171027 and 11361020).  This project is also partially supported
by the Specialized Research Fund for the Doctoral Program of Higher Education
of China (Grant No. 20120003110003)  and the Fundamental Research Funds for Central
Universities of China (Grant Nos. 2013YB60 and 2014KJJCA10).}}
\author{Dachun Yang and Junqiang Zhang\,\footnote{Corresponding author}}
\date{ }
\maketitle

\vspace{-0.8cm}

\begin{center}
\begin{minipage}{13.8cm}
{\small {\bf Abstract}\quad
Let $w$ be a Muckenhoupt $A_2(\mathbb{R}^n)$ weight
and $L_w:=-w^{-1}\mathop\mathrm{div}(A\nabla)$
the degenerate elliptic operator on the Euclidean space $\mathbb{R}^n$, $n\geq 2$.
In this article, the authors
establish some weighted $L^p$ estimates of Kato square roots associated to
the degenerate elliptic operators $L_w$. More precisely, the authors
prove that, for $w\in A_{p}(\mathbb{R}^n)$,
$p\in(\frac{2n}{n+1},\,2]$ and any $f\in C^\infty_c(\mathbb{R}^n)$,
$\|L_w^{1/2}(f)\|_{L^p(w,\,\mathbb{R}^n)}
\sim \|\nabla f\|_{L^p(w,\,\mathbb{R}^n)}$,
where $C_c^\infty(\mathbb{R}^n)$ denotes the set of all infinitely differential functions with
compact supports.}
\end{minipage}
\end{center}

\section{Introduction}\label{s1}

\hskip\parindent
The Kato square root problem, which has a long history, was originally posed by Kato \cite{Ka61} in 1961.
It  amounts to identifying the domain of the square root of an abstract maximal
accretive operator as the domain of the corresponding sesquilinear form. Although it is known that
this problem has an affirmative answer in a few particular cases, in general, the Kato square root problem
does not hold true; see, for example, \cite{lion62,m72} for some counterexamples. However,
by noticing
that Kato posed his problem with the motivation from a special case of elliptic differential operators,
McIntosh \cite{m84,m82} refined the statement of the Kato square root problem in the setting of elliptic operators.
More precisely,  let $L:=-\div(A\nabla)$ be the second order elliptic operator on $\rn$, with $A$ being an
$n\times n$ matrix of complex bounded measurable
functions on $\rn$ satisfying the  elliptic condition. The refined formulation of the Kato square root problem by
McIntosh consists in showing that the domain of the square root $L^{1/2}$ coincides with the Sobolev space
$W^{1,2}(\rn)$ and
\begin{eqnarray}\label{eq in0}
\lf\|L^{1/2}(f)\r\|_{L^2(\rn)}\sim\|\nabla f\|_{L^2(\rn)}
\end{eqnarray}
with the equivalent positive constants independent of $f$.
This problem was completely solved by Auscher et al. \cite{ahlmt02,ahlt01,hlm02}
in the past decade, which consists one of the most celebrated results in harmonic analysis of recent years.
For a more complete history of this problem, we refer the reader to the above
papers or to the review by Kenig \cite{Ke02} and their references.

Observe that \eqref{eq in0} consists in comparing the $L^2$ norms of
$L^{1/2}(f)$ and $\nabla f$. For a general $p\in (1,\,\fz)$,
the $L^p$ theory of square roots has also attracted considerable attention
(see \cite{Au07,HM03} and the references cited therein). In particular,
Auscher \cite{Au07} showed that, for any $f\in C_c^\fz(\rn)$,
\begin{eqnarray}\label{eq in1}
\lf\|L^{1/2}(f)\r\|_{L^p(\rn)}\sim \|\nabla f\|_{L^p(\rn)},\ \ \
p\in\lf(p_-(L),\,2+\varepsilon(L)\r),
\end{eqnarray}
here and hereafter, the equivalent positive constants in \eqref{eq in1} are independent of $f$,
$C_c^\fz(\rn)$ denotes the set of all infinitely
differential functions with compact supports,
$p_-(L):=\inf\{p\in[1,\,\fz]:\ \nabla L^{-1/2}:\ L^p(\rn)\to L^p(\rn)\}\in[1,\,\frac{2n}{n+2})$
and $\varepsilon(L)$
is a positive constant depending on $L$. Moreover,
Hofmann et al. \cite{HMM11} generalized the aforementioned result to the range
$p\in (\frac{p_-(L)n}{n+p_-(L)},\,2+\varepsilon(L))$,
by establishing the Riesz transform
characterizations of the Hardy spaces $H_L^p(\rn)$
associated to the second order elliptic operator $L=-\div(A\nabla)$, namely,
for all $f\in H_L^p(\rn)$,
\begin{eqnarray}\label{eq in2}
\|f\|_{H_L^p(\rn)}\sim \lf\|\nabla L^{-1/2}(f)\r\|_{H^p(\rn)},\ \ \
p\in\lf(\frac{p_-(L)n}{n+p_-(L)},\,2+\varepsilon(L)\r),
\end{eqnarray}
where $H^p(\rn)$ denotes the classical Hardy space
and the equivalent positive constants in \eqref{eq in2} are independent of $f$.
Noticing that, for all $p\in(p_-(L),\,2+\varepsilon(L))$,
both $H^p(\rn)$ and $H^p_L(\rn)$ coincide with the Lebesgue spaces $L^p(\rn)$
(see \cite[Proposition 9.1(v)]{HMM11}),
thus \eqref{eq in2} covers \eqref{eq in1}.

In the present article, we consider the $L^p$ theory of square roots
in the case of degenerate elliptic operators.
To be precise, let $w\in A_2(\mathbb{R}^n)$ be a Muckenhoupt weight.
A matrix $A(x):=(A_{ij}(x))_{i,j=1}^n$ of complex-valued, measurable
functions on $\rn$ is said to satisfy the \emph{degenerate elliptic condition}
if there exist positive constants $\lz\le\blz$ such that,
for almost every $x\in\rn$ and all $\xi$, $\eta\in\cc^n$,
\begin{eqnarray}\label{degenerate C1}
\lf|\langle A(x) \xi,\,\eta \rangle\r| \le \blz w(x)|\xi||\eta|
\end{eqnarray}
and
\begin{eqnarray}\label{degenerate C2}
\Re \langle A(x) \xi,\,\xi \rangle \geq \lz w(x)|\xi|^2,
\end{eqnarray}
where $\Re z$ denotes the \emph{real part} of $z$ for any $z\in\cc$.
For such a matrix $A(x)$,
the associated \emph{degenerate elliptic operator} $L_w$ is defined by setting,
for all $f\in D(L_w)\st\mathcal{H}_0^1(w,\,\rn)$,
\begin{eqnarray}\label{Lw}
L_wf:=-\frac{1}{w}\div (A\nabla f),
\end{eqnarray}
which is interpreted in the usual weak sense via the sesquilinear form,
where $D(L_w)$ denotes the domain of $L_w$.
Here and hereafter, $\mathcal{H}_0^1(w,\,\rn)$ denotes the \emph{weighted Sobolev space} which is
defined to be the closure of $C_c^\fz(\rn)$ with respect to the \emph{norm}
$$\|f\|_{\mathcal{H}_0^1(w,\,\rn)}:=\lf\{\int_\rn \lf[|f(x)|^2+|\nabla f(x)|^2\r]w(x)\,dx\r\}^{1/2}.$$
The \emph{sesquilinear form} $\mathfrak{a}$ associated with $L_w$ is defined by setting,
for all $f$, $g\in\mathcal{H}_0^1(w,\,\rn)$,
\begin{eqnarray}\label{eq sesqui form}
\mathfrak{a}(f,\,g):=\dint_{\rn}[A(x)\nabla f(x)]\cdot \overline{\nabla g(x)}\,dx.
\end{eqnarray}
Operators of the form \eqref{Lw} and the associated elliptic equations
were first studied by Fabes et al. \cite{FKS82}
and have also been considered by a number of other authors (see, for example,
\cite{CF84,CS85,CF87} and, especially, some recent articles by Cruz-Uribe et al.
\cite{CR08,CR12,CR13,CR14}).

Observe that the accretive condition \eqref{degenerate C2} enables one to define
the square root $L^{1/2}_w$ (see \cite{Ka61,Ka95}).
It is a natural question to consider the associated Kato square root problem
in the case of the degenerate elliptic operator $L_w$. In particular,
Cruz-Uribe et al. \cite{CR13} proved that, for any $f\in H_0^1(w,\,\rn)$,
\begin{eqnarray*}
\lf\|L_w^{1/2}(f)\r\|_{L^2(w,\,\rn)}\sim \|\nabla f\|_{L^2(w,\,\rn)},
\end{eqnarray*}
where $L^2(w,\,\rn)$ denotes the \emph{weighted Lebesgue space} with the norm
\begin{eqnarray*}
\|f\|_{L^2(w,\,\rn)}:=\lf[\int_\rn |f(x)|^2w(x)\,dx\r]^{\frac{1}{2}}.
\end{eqnarray*}
This result solves the Kato square root problem associated to the operator $L_w$.
Notice that, when $w\equiv 1$, $L_w$ is just the second elliptic operator $L$,
thus, the results in \cite{CR13} may be seen as generalizations of those in \cite{ahlmt02}.

Motivated by the aforementioned results in \cite{CR13,ahlmt02,Au07,HMM11},
our aim of this article is to study the weighted $L^p$ estimates of Kato
square roots associated to $L_w$.
To be precise, let $w\in A_\fz(\rn)$ be a Muckenhoupt weight
(see \eqref{eq A-p}, \eqref{eq A-1} and \eqref{eq A-fz} below
for the precise definitions of $A_p(\rn)$ of Muckenhoupt weights with $p\in[1,\,\fz]$).
For any measurable set $E$ of $\rn$ and $p\in(0,\,\fz)$,
$L^p(w,\,E)$ denotes the \emph{weighted Lebesgue space} with the \emph{(quasi-)norm}
\begin{eqnarray*}
\|f\|_{L^p(w,\,E)}:=\lf\{\dint_E\lf|f(x)\r|^pw(x)\,dx\r\}^{\frac1p}.
\end{eqnarray*}
Let $L_w$ be a degenerate elliptic operator
as in \eqref{Lw} with $w\in A_2(\rn)$.
The following theorem is the main result of
the present article, which is proved in Section \ref{s6}.
\begin{thm}\label{cor main}
Let $p\in(\frac{2n}{n+1},\,2]$ and $w\in A_p(\rn)$.
Then there exists a positive constant $C$ such that,
for any $f\in C_c^\fz(\rn)$,
\begin{eqnarray*}
C^{-1}\|\nabla f\|_{L^p(w,\,\rn)}\le \lf\|L_w^{1/2}(f)\r\|_{L^p(w,\,\rn)}
\le C\|\nabla f\|_{L^p(w,\,\rn)}.
\end{eqnarray*}
\end{thm}
This result establishes the weighted $L^p$ estimates of Kato square roots associated to
the degenerate elliptic operators $L_w$ for $p\in(\frac{2n}{n+1},\,2]$.
In particular, when $p=2$, Theorem \ref{cor main},
together with a density argument, leads to the corresponding result in \cite{CR13}.

To prove Theorem \ref{cor main}, we use the strategy of
establishing the Riesz transform characterizations of the Hardy spaces associated
to $L_w$, which is accomplished by Propositions \ref{thm equi}, \ref{thm main1} and \ref{thm main} below.
We point out that this idea is
inspired by Hofmann et al. \cite{HMM11}.
Now we introduce some related definitions and notation on the Hardy spaces
associated to the degenerate elliptic operator $L_w$.
In what follows, let $\rr_+^{n+1}:=\rr^n\times(0,\fz)$.
Let $w\in A_2(\rn)$ and $L_w$ be as in \eqref{Lw},
for any $f\in L^2(w,\,\rn)$ and $x\in\rn$, the \emph{square function $\cs_{L_w}(f)$
associated with} $L_w$ is defined by setting
\begin{eqnarray*}
\cs_{L_w}(f)(x):=\lf[\iint_{\bgz(x)}\lf|t^2L_we^{-t^2L_w}(f)(y)\r|^2w(y)\,
\frac{dy}{w(B(x,t))}\,\frac{dt}{t}\r]^{1/2},
\end{eqnarray*}
where $B(x,t):=\{y\in\rn:\ |x-y|<t\}$,
$w(B(x,t)):=\int_{B(x,t)}w(y)\,dy$
and
\begin{eqnarray}\label{cone}
\bgz_\az(x):=\{(y,t)\in\rr^{n+1}_+:\ |x-y|<\az t\}
\end{eqnarray}
denotes the \emph{cone of aperture} $\az$ \emph{with
vertex} $x$. In particular, if $\az=1$, we write $\Gamma(x)$
instead of $\Gamma_\az(x)$.

For any $p\in(0,\,\fz)$,
the Hardy space $H_{L_w}^p(\rn)$ associated to $L_w$ is defined as follows.
\begin{defn}\label{def Hardy space}
Let $w\in A_2(\rn)$
and $L_w$ be the degenerate elliptic operator
as in \eqref{Lw} with the matrix $A$ satisfying the degenerate elliptic
conditions \eqref{degenerate C1} and \eqref{degenerate C2}.
For any $p\in(0,\,2]$,
the \emph{Hardy space} $H_{L_w}^p(\rn)$, \emph{associated to}
$L_w$, is defined as the completion of the space
\begin{eqnarray*}
\lf\{f\in L^2(w,\,\rn):\ \|S_{L_w}(f)\|_{L^p(w,\,\rn)}<\fz\r\}
\end{eqnarray*}
with respect to the \emph{(quasi-)norm}
\begin{eqnarray*}
\lf\|f\r\|_{H_{L_w}^p(\rn)}:=\|S_{L_w}(f)\|_{L^p(w,\,\rn)}.
\end{eqnarray*}
For any $p\in(2,\,\fz)$, define
\begin{eqnarray*}
H_{L_w}^p(\rn):=\lf(H_{L_w^\ast}^{p'}(\rn)\r)^\ast,
\end{eqnarray*}
here and hereafter, $1/p+1/p'=1$ and $L_w^\ast$ is the adjoint operator of $L_w$
in $L^2(w,\,\rn)$.
\end{defn}

We point out that the study of the Hardy spaces associated to different operators
(for example, the non-negative self-adjoint operator, the second order elliptic operator
$-\div (A\nabla)$ and the Schr\"odinger operator $-\Delta +V$) has
attracted considerable attention and the real-variable theory of these spaces
has been established in recent years (see, for example, \cite{ADM05,DY05,DY051,AMR08,Yan08,JY10,HM09,
HMM11,dl13,dhmmy13,ccyy13}).

Moreover, we need to introduce the following Hardy space $H_{L_w,\,{\rm Riesz}}^p(\rn)$
associated to the Riesz transform $\nabla L_w^{-1/2}$,
which, when $w\equiv1$, is a special case of that defined in \cite[p.\,728]{HMM11}.

\begin{defn}\label{def Riesz Hardy}
Let $p\in(1,\,\fz)$, $w\in A_2(\rn)$
and $L_w$ be the degenerate elliptic operator
as in \eqref{Lw} with the matrix $A$ satisfying the degenerate elliptic
conditions \eqref{degenerate C1} and \eqref{degenerate C2}.
The \emph{Hardy space} $H_{L_w,\,{\rm Riesz}}^p(\rn)$ is defined as the completion of the space
\begin{eqnarray*}
\lf\{f\in L^2(w,\,\rn):\ \nabla L_w^{-1/2}(f)\in L^p(w,\,\rn)\r\}
\end{eqnarray*}
with respect to the \emph{norm}
\begin{eqnarray*}
\lf\|f\r\|_{H_{L_w,\,{\rm Riesz}}^p(\rn)}:=\lf\|\nabla L_w^{-1/2}(f)\r\|_{L^p(w,\,\rn)}.
\end{eqnarray*}
\end{defn}

\begin{rem}\label{rem 2}
Comparing with Definition \ref{def Riesz Hardy}, recall that, 
when $p\in(0,\,1]$ and $w\in A_2(\rn)$, the Hardy space $H_{L_w,\,{\rm Riesz}}^p(\rn)$
was introduced in \cite[Definition 1.2]{yz15},
which is defined as the completion of the space
\begin{eqnarray*}
\lf\{f\in L^2(w,\,\rn):\ \nabla L_w^{-1/2}(f)\in H_w^p(\rn)\r\}
\end{eqnarray*}
with respect to the \emph{quasi-norm}
\begin{eqnarray*}
\lf\|f\r\|_{H_{L_w,\,{\rm Riesz}}^p(\rn)}:=\lf\|\nabla L_w^{-1/2}(f)\r\|_{H_w^p(\rn)},
\end{eqnarray*}
where $H_w^p(\rn)$ denotes the classical weighted Hardy space.
Moreover, in \cite{yz15}, the Hardy spaces $H_{L_w,\,{\rm Riesz}}^p(\rn)$ and $H^p_{L_w}(\rn)$
were proved to coincide when $p\in(\delta,\,1]$,
where $\delta\in(0,\,1)$ is some fixed constant.
\end{rem}

To prove Theorem \ref{cor main},
we first prove the following three propositions.

\begin{prop}\label{thm equi}
Let $w\in A_2(\rn)$. Then, for any given $p\in(\frac{2n}{n+1},\,\frac{2n}{n-1})$,
$H_{L_w}^p(\rn)$ and $L^p(w,\,\rn)$ coincide with equivalent norms.
\end{prop}

Proposition \ref{thm equi} is proved in Section \ref{s3.x}.
In particular, when $w\equiv 1$, $L_w$ is just
the usual second order elliptic operator $L=-\div (A\nabla)$ studied in
\cite{HMM11}, where Hofmann et al. proved that, for any $p\in(p_{-}(L),\,p_{+}(L))$,
$H_{L}^p(\rn)$ coincides with the Lebesgue space $L^p(\rn)$ with equivalent norms.
Notice that $1\le p_{-}(L)<\frac{2n}{n+1}<\frac{2n}{n-1}<p_{+}(L)\le\fz$
(see \cite[p.\,4]{HMM11}). Recall that, via the local weighted Sobolev embedding inequality proved in \cite{FKS82}
(see also Lemma \ref{lem imbedding} below),
it was proved in \cite[Proposition 1.5]{ZCJY14} that, for any $\frac{2n}{n+1}\le p\le q\le\frac{2n}{n-1}$,
the semigroup $\{e^{-tL_w}\}_{t\geq 0}$ satisfies the
weighted $L^p-L^q$ off-diagonal estimates on balls
(see also Proposition \ref{pro ODEB} below). This is a main tool used in
the proof of Proposition \ref{thm equi}, which restricts
the range of $p$ in Proposition \ref{thm equi} to the narrower interval
$(\frac{2n}{n+1},\,\frac{2n}{n-1})$ instead of $(p_{-}(L),\,p_{+}(L))$.
It is still unclear whether Proposition \ref{thm equi} still holds true or not
for a wider range of $p$ than $(\frac{2n}{n+1},\,\frac{2n}{n-1})$.

\begin{prop}\label{thm main1}
{\rm (i)} Let $w\in A_2(\rn)$ and $p\in(\frac{2n}{n+1},\,2]$. Then there exists a
positive constant $C$ such that, for any $f\in H_{L_w}^p(\rn)$,
\begin{eqnarray*}
\lf\|\nabla L_w^{-1/2}(f)\r\|_{L^p(w,\,\rn)}\le C\|f\|_{H_{L_w}^p(\rn)}.
\end{eqnarray*}

{\rm (ii)} Let $w\in A_q(\rn)$ with $q\in [1,\,1+\frac{1}{n})$ and $p\in [1,\,\frac{2n}{n+1}]$.
Then there exists a positive constant
$C$ such that, for any $f\in H_{L_w}^p(\rn)$,
\begin{eqnarray*}
\lf\|\nabla L_w^{-1/2}(f)\r\|_{L^p(w,\,\rn)}\le C\|f\|_{H_{L_w}^p(\rn)}.
\end{eqnarray*}
\end{prop}

Proposition \ref{thm main1} is proved in Section \ref{s4}.
We point out that, when $w\equiv 1$, Proposition \ref{thm main1}
is covered by \cite[Propositions 5.32 and 5.6]{HMM11}, where
Hofmann et al. proved that, for any $p\in [1,\,2+\varepsilon(L))$,
$\nabla L^{-1/2}$ is bounded from $H_L^p(\rn)$ to $L^p(\rn)$.
We prove Proposition \ref{thm main1} by using
the local weighted Poincar\'{e} inequality in \cite{FKS82}
(see also Lemma \ref{lem poincare} below),
which implies that, for $w\in A_2(\rn)$ and any $p\in(\frac{2n}{n+2},\,2]$,
$\nabla L_w^{-1/2}$ is bounded on $L^p(w,\,\rn)$. This restricts the results
of Proposition \ref{thm main1} to the narrower interval $p\in [1,\,2]$
instead of $p\in [1,\,2+\varepsilon(L))$.

\begin{prop}\label{thm main}
Let $n\geq 2$, $p\in(\frac{2n}{n+1},\,\frac{2n}{n-1})$
and $w\in A_p(\rn)\cap A_2(\rn)$.
Then there exists a positive constant $C$ such that, for any
$f\in L^2(w,\,\rn)\cap H_{L_w,{\rm Riesz}}^p(\rn)$,
\begin{eqnarray*}
\|f\|_{H_{L_w}^p(\rn)}\le C\lf\|\nabla L_w^{-1/2}(f)\r\|_{L^p(w,\,\rn)}.
\end{eqnarray*}
\end{prop}

Proposition \ref{thm main} is proved in Section \ref{s5}.
Its proof relies on the weighted off-diagonal
estimates on balls for $L_w$ (see Proposition \ref{pro ODEB} below)
and the local weighted Poincar\'{e} and Sobolev embedding inequalities
(see Lemmas \ref{lem poincare} and \ref{lem imbedding} below),
which restrict the results of Proposition \ref{thm main}
to $n\geq 2$, $p\in(\frac{2n}{n+1},\,\frac{2n}{n-1})$
and $w\in A_p(\rn)$ when $p\le2$.
Proposition \ref{thm main} is an analogue of
\cite[Proposition 5.34]{HMM11}, where Hofmann et al. proved that,
if, for some $r\in(1,\,2]$, the semigroup
$\{e^{-tL}\}_{t\geq 0}$ satisfies $L^r-L^2$ off-diagonal estimates, then,
for any $p\in (\max\{1,\,\frac{rn}{n+r}\},\,p_+(L))$,
there exists a positive constant $C$ such that, for
any $h\in L^2(\rn)\cap H_{L,\,{\rm Riesz}}^p(\rn)$,
$$\|h\|_{H_L^p(\rn)}\le C\lf\|\nabla L^{-1/2}(h)\r\|_{L^p(\rn)}.$$
Notice that, for any $r\in(1,\,2]$,
$(\frac{2n}{n+1},\,\frac{2n}{n-1})\st(\max\{1,\,\frac{rn}{n+r}\},\,p_+(L))$.
Thus, when $w\equiv1$, Proposition \ref{thm main} is
covered by \cite[Proposition 5.34]{HMM11}.

The rest of this article is organized as follows. In Subsection \ref{s2.0}, we first recall
some notions and results on Muckenhoupt weights;
in Subsection \ref{s2.1}, we recall the holomorphic functional calculus of $L_w$;
then, in Subsection \ref{s2.2},
we introduce the weighted off-diagonal estimates for $L_w$, which have been
established in \cite{ZCJY14}; in Subsection \ref{s2.3},
we recall the notion of the weighted tent space and recall some results on their dual and interpolation results.
In Section \ref{s3}, by following the strategy used in \cite[Section 4]{HMM11},
for $p\in(0,\,\fz)$, we establish the square function characterizations of $H_{L_w}^p(\rn)$
(see Propositions \ref{pro 1} and \ref{pro 2} below).

We end this section by making some conventions on notation. Throughout this article,
$L_w$ always denotes a degenerate elliptic operator as in \eqref{Lw}.
We denote by $C$ a positive constant which is independent of the main parameters,
but it may vary from line to line.
We also use $C_{(\az, \bz,\ldots)}$ to denote a positive constant depending on
the parameters $\az$, $\bz$, $\ldots$.
The \emph{symbol $f\ls g$} means that $f\le Cg$.
If $f\ls g$ and $g\ls f$, then we write $f\sim g$.
For any measurable subset $E$ of $\rn$, we denote by $E^\com$ the \emph{set $\rn\bh E$}.
Let $\nn:=\{1,\,2,\,\ldots\}$ and $\zz_+:=\nn\cup\{0\}$.
For any closed set $F\st\rn$, we let
\begin{equation}\label{eq tent}
R(F):=\bigcup_{x\in F}\bgz(x),
\end{equation}
where $\bgz(x)$, for all $x\in F$, is as in \eqref{cone} with $\az=1$.
For any $\mu\in(0,\,\pi)$, let
\begin{eqnarray}\label{eq sigma}
\Sigma_{\mu}^0:=\{z\in\mathbb{C}\setminus\{0\}:\ |\arg z|<\mu\}.
\end{eqnarray}
For any ball $B:=(x_B,r_B)\st\rn$ with $x_B\in\rn$ and $r_B\in(0,\,\fz)$, $\az\in(0,\fz)$ and $j\in\nn$,
we let $\az B:=B(x_B,\az r_B)$,
\begin{eqnarray}\label{eq-def of ujb}
U_0(B):=B\ \ \ \text{and}\ \ \ U_j(B):=(2^jB)\setminus (2^{j-1}B).
\end{eqnarray}
For any $p\in[1,\,\fz)$, $p'$ denotes its conjugate number,
namely, $1/p+1/p'=1$.

\section{Preliminaries}\label{s2}
\hskip\parindent
In this section, we first recall the definition of the Muckenhoupt weights and
some of their properties. Then we recall the holomorphic
functional calculus of $L_w$, as introduced by McIntosh \cite{M86},
and the weighted off-diagonal estimates on balls for $L_w$.
Finally, we introduce the weighted tent spaces and some of their properties.
\subsection{Muckenhoupt weights}\label{s2.0}

\hskip\parindent
Let $q\in[1,\fz)$. A nonnegative and locally integrable function $w$ on $\rn$
is said to belong to the \emph{Muckenhoupt class} $A_q(\rn)$
if there exists a positive constant $C$ such that, for any ball $B\st\rn$, when $q\in(1,\fz)$,
\begin{eqnarray}\label{eq A-p}
\frac{1}{|B|}\int_B w(x)\,dx\lf\{\frac{1}{|B|}\int_B [w(x)]^{-\frac{1}{q-1}}\,dx\r\}^{q-1}\le C
\end{eqnarray}
or, when $q=1$,
\begin{eqnarray}\label{eq A-1}
\frac{1}{|B|}\int_B w(x)\,dx\le C\essinf_{x\in B}w(x).
\end{eqnarray}
We also let
\begin{eqnarray}\label{eq A-fz}
A_\fz(\rn):=\bigcup_{q\in[1,\fz)}A_q(\rn)
\end{eqnarray}
and $w(E):=\int_E w(x)\,dx$
for any measurable set $E\st\rn$.

Let $r\in(1,\fz]$. A nonnegative locally integrable function $w$ is said to
belong to the \emph{reverse H\"older class $RH_r(\rn)$} if there exists
a positive constant $C$ such that, for any ball $B\st\rn$,
\begin{eqnarray*}
\lf\{\frac{1}{|B|}\int_B [w(x)]^r\,dx\r\}^{1/r}\le C\frac{1}{|B|}\int_B w(x)\,dx,
\end{eqnarray*}
where we replace  $\{\frac{1}{|B|}\int_B [w(x)]^r\,dx\}^{1/r}$ by $\|w\|_{L^\infty(B)}$ when $r=\infty$.

We recall some properties of Muckenhoupt weights and reverse H\"older classes
in the following two lemmas (see, for example, \cite{Du01} for their proofs).
\begin{lem}\label{lem Ap-1}
{\rm(i)} If $1\le p\le q\le\fz$, then $A_1(\rn)\st A_p(\rn)\st A_q(\rn)$.

{\rm(ii)} $A_\fz(\rn):=\cup_{p\in[1,\fz)}A_p(\rn)=\cup_{r\in(1,\fz]}RH_r(\rn)$.
\end{lem}

\begin{lem}\label{lem Ap-2}
Let $q\in[1,\fz)$ and $r\in (1,\fz]$. If a nonnegative measurable function
$w\in A_q(\rn)\cap RH_r(\rn)$, then there exists a constant $C\in(1,\fz)$
such that, for any ball $B\st\rn$ and any measurable subset $E$ of $B$,
$$C^{-1}\lf(\frac{|E|}{|B|}\r)^q\le\frac{w(E)}{w(B)}\le C\lf(\frac{|E|}{|B|}\r)^{\frac{r-1}{r}}.$$
\end{lem}

\subsection{Holomorphic functional calculi for $L_w$}\label{s2.1}
\hskip\parindent
Let $L_w$ be the degenerate elliptic operator
as in \eqref{Lw} with the matrix $A$ satisfying the degenerate elliptic
conditions \eqref{degenerate C1} and \eqref{degenerate C2}.
By \cite[pp.\,291-294]{CR08}, we know that $L_w$ is an
operator of \emph{type $\omega$} with $\omega:=\arctan(\Lambda/\lambda)\in(0,\,\pi/2)$
(see \cite{M86} for the denfinition),
where $0<\lambda\le\Lambda<\fz$
are as in \eqref{degenerate C1} and \eqref{degenerate C2}, and
$-L_w$ generates a holomorphic semigroup in the sector $\Sigma_{\pi/2-\omega}^0$,
where $\Sigma_{\pi/2-\omega}^0$ is as in \eqref{eq sigma} with $\mu$ replaced by
$\pi/2-\omega$.

Furthermore, $L_w$ has a bounded holomorphic functional calculus on $L^2(w,\,\rn)$
as defined by McIntosh \cite{M86} (see also \cite[Lecture 4]{ADM96}).
We now recall some preliminary definitions.

For any $\mu\in(0,\,\pi/2)$, define
\begin{eqnarray*}
H_\fz(\Sigma_{\mu}^0):=\lf\{f:\ \Sigma_{\mu}^0\to\mathbb{C} \text{ is holomorphic and }
\|f\|_{L^\fz(\Sigma_{\mu}^0)}<\fz\r\}.
\end{eqnarray*}
For any $\alpha,\beta\in(0,\,\fz)$, let
\begin{eqnarray}\label{eq 2.0}
\Psi_{\az,\,\bz}(\Sigma_{\mu}^0):=\lf\{\psi\in H_\fz(\Sigma_{\mu}^0):\
|\psi(z)|\le C\frac{|z|^\az}{1+|z|^{\az+\bz}},\ \forall z\in\Sigma_{\mu}^0\r\},
\end{eqnarray}
where $C$ is a positive constant independent of $z\in\Sigma_\mu^0$.

Let $\Psi(\Sigma_{\mu}^0):=\cup_{\az,\bz\in(0,\,\fz)}\Psi_{\az,\,\bz}(\Sigma_{\mu}^0)$.
For any $\mu\in(\omega,\,\pi/2)$ and $\psi\in\Psi(\Sigma_{\mu}^0)$,
define
\begin{eqnarray}\label{eq def1}
\psi(L_w):=\frac{1}{2\pi i}\int_{\gamma}\psi(\zeta)(\zeta I-L_w)^{-1}\,d\zeta,
\end{eqnarray}
where $\gamma:=\{re^{i\nu}:\ r\in(0,\,\fz)\}\cup\{re^{-i\nu}:\ r\in(0,\,\fz)\}$,
$\nu\in(\omega,\,\mu)$, is a curve consisting of two rays parameterized anti-clockwise.

In general, for any $\psi\in H_\fz(\Sigma_\mu^0)$, $\psi(L_w)$ can be defined
by a limiting procedure (see \cite[Theorem G]{ADM96}).

\subsection{Weighted off-diagonal estimates on balls for $L_w$}\label{s2.2}

\hskip\parindent
The notion of weighted off-diagonal estimates on balls was first
introduced by Auscher and Martell in \cite{am07ii}.

\begin{defn}[\cite{am07ii}]\label{def ODEB}
Let $p,\,q\in[1,\fz]$ with $p\le q$, $w\in A_\fz(\rn)$ and $\{T_t\}_{t>0}$ be a family of sublinear
operators. The family $\{T_t\}_{t>0}$ is said to satisfy \emph{weighted $L^p$-$L^q$
off-diagonal estimates on balls}, denoted by $T_t\in \mathcal{O}_{w}(L^p-L^q)$, if
there exist constants $\theta_1,\,\theta_2 \in[0,\fz)$ and $C,\,c\in(0,\fz)$ such that,
for any $t\in(0,\fz)$, any ball $B:=B(x_B,r_B)\subset\rn$ with $x_B\in\rn$ and $r_B\in (0,\fz)$,
and any $f\in L^p_{\loc}(w,\,\rn)$,
\begin{eqnarray}\label{ODEB1}
&&\lf\{\frac{1}{w(B)}\dint_B|T_t\lf(\chi_Bf\r)(x)|^qw(x)\,dx\r\}^{1/q}\\
&&\hs\le C
\lf[\Upsilon\lf(\frac{r_B}{t^{1/2}}\r)\r]^{\theta_2}\lf\{\frac{1}{w(B)}\dint_B
|f(x)|^pw(x)\,dx\r\}^{1/p}\noz
\end{eqnarray}
and, for any $j\in\nn\cap [3,\,\fz)$,
\begin{eqnarray*}
&&\lf\{\frac{1}{w(2^jB)}\dint_{U_j(B)}|T_t\lf(\chi_{B}f\r)(x)|^qw(x)\,dx\r\}^{1/q}\\ \nonumber
&&\hs\le C
2^{j\theta_1}\lf[\Upsilon\lf(\frac{2^jr_B}{t^{1/2}}\r)\r]^{\theta_2}
e^{-c\frac{(2^jr_B)^{2}}{t}}\lf\{\frac{1}{w(B)}\dint_B
|f(x)|^pw(x)\,dx\r\}^{1/p}
\end{eqnarray*}
and
\begin{eqnarray}\label{ODEB3}
&&\lf\{\frac{1}{w(B)}\dint_B |T_t(\chi_{U_j(B)}f)(x)|^qw(x)\,dx\r\}^{1/q}\\ \nonumber
&&\hs\le C 2^{j\theta_1}\lf[\Upsilon\lf(\frac{2^jr_B}{t^{1/2}}\r)\r]^{\theta_2}
e^{-c\frac{(2^jr_B)^{2}}{t}}\lf\{\frac{1}{w(2^jB)}\dint_{U_j(B)}
|f(x)|^pw(x)\,dx\r\}^{1/p},
\end{eqnarray}
where $U_j(B)$ is as in \eqref{eq-def of ujb} and, for all $s\in(0,\fz)$,
$\Upsilon(s):=\max\{s,\frac{1}{s}\}$.
\end{defn}

The following proposition is just \cite[Proposition 1.5]{ZCJY14}.
\begin{prop}[\cite{ZCJY14}]\label{pro ODEB}
Let $w\in A_2(\rn)$ and $k\in\zz_+$.
Then, for any $\frac{2n}{n+1}\le p\le q\le\frac{2n}{n-1}$,
$(tL_w)^ke^{-tL_w}\in\mathcal{O}_{w}(L^p-L^q)$.
\end{prop}

The following lemma is an analogue of \cite[Lemma 2.40]{HMM11},
whose proof being omitted.
\begin{lem}\label{lem 2.0}
Let $\omega:=\arctan(\Lambda/\lambda)$, $\mu\in(\omega,\,\pi/2)$
and $\sigma_1,\,\sigma_2,\,\tau_1,\,\tau_2\in(0,\,\fz)$.
Assume that $\psi\in\Psi_{\sigma_1,\tau_1}(\Sigma_\mu^0)$,
$\wz\psi\in\Psi_{\sigma_2,\tau_2}(\Sigma_\mu^0)$ and $f\in H_\fz(\Sigma_\mu^0)$.
Then, for any $a\in(0,\,\min\{\sigma_1,\,\tau_2\})$ and $b\in(0,\,\min\{\sigma_2,\,\tau_1\})$,
there exists a family of sub-linear operators, $\{T_{s,t}\}_{s,t>0}$, such that
\begin{eqnarray*}
\psi(L_w)\circ f(L_w)\circ \wz\psi(L_w)=\min\lf\{\lf(\frac{s}{t}\r)^a,\,
\lf(\frac{t}{s}\r)^b\r\}T_{s,t},
\end{eqnarray*}
where $\{T_{s,t}\}_{s,t>0}$ have the following properties:

(i) There exists a positive constant $C$, independent of $s$,
such that, for any $t\in[s,\,\fz)$, any closed sets $E$ and $F$ of
$\rn$ and $f\in L^2(w,\,\rn)$ with $\supp f\st E$,
\begin{eqnarray*}
\|T_{s,t}(f)\|_{L^2(w,\,F)}\le C\|f\|_{L^\fz(\Sigma_\mu^0)}
\lf[\min\lf\{1,\,\frac{t}{[d(E,\,F)]^2}\r\}\r]^{\sigma_2+a}\|f\|_{L^2(w,\,E)}.
\end{eqnarray*}

(ii) There exists a positive constant $C$, independent of $t$,
such that, for any $s\in[t,\,\fz)$, any closed sets $E$ and $F$ of
$\rn$ and $f\in L^2(w,\,\rn)$ with $\supp f\st E$,
\begin{eqnarray*}
\|T_{s,t}(f)\|_{L^2(w,\,F)}\le C\|f\|_{L^\fz(\Sigma_\mu^0)}
\lf[\min\lf\{1,\,\frac{s}{[d(E,\,F)]^2}\r\}\r]^{\sigma_1+b}\|f\|_{L^2(w,\,E)}.
\end{eqnarray*}

\end{lem}

\subsection{Weighted tent spaces}\label{s2.3}
\hskip\parindent
Let $w\in A_\fz(\rn)$ and $f$ be a measurable function on $\rr^{n+1}_+$.
For any $x\in\rn$, define
\begin{equation*}
A(f)(x):=\lf[\iint_{\bgz(x)}|f(y,t)|^2\,\frac{w(y)\,dy\,dt}{w(B(x,t))t}\r]^{1/2},
\end{equation*}
where $\bgz(x)$ is as in \eqref{cone} with $\az=1$.
For any $p\in(0,\fz)$, the \emph{weighted tent space}
$T^p(w,\,\rn)$ is defined to be the space of all measurable functions $f$
on $\rr_+^{n+1}$
such that $\|f\|_{T^p(w,\,\rn)}:=\|A(f)\|_{L^p(w,\,\rn)}<\fz$.

For any open set $O\st\rn$, the tent over $O$ is defined by
$$\wh O:=\{(x,t)\in\rnn:\ \dist(x,\, O^\com)\geq t\}.$$

Let $p\in (0,1]$ and $w\in A_\fz(\rn)$. A measurable function $a$ on $\rr_+^{n+1}$
is called a \emph{$(w,\,p,\,2)$-atom} if there exists a ball $B$ of $\rn$ such that

(i) $\supp a\subset \widehat{B}$;

(ii)
\begin{equation}\label{eq tent atom}
\lf[\int_0^\fz\int_\rn|a(y,t)|^2w(y)\,dy\,\frac{dt}{t}\r]^\frac{1}{2}
\le [w(B)]^{\frac{1}{2}-\frac{1}{p}}.
\end{equation}

Noticing that, for any $w\in A_\fz(\rn)$, $(\rn,\,|\cdot|,\,w(x)dx)$ is a space of
homogeneous type in the sense of Coifman and Weiss \cite{CW71,CW77},
the following lemma was proved in \cite[Theorem 1.1]{Ru07},
except for the last part concerning the $T^2(w,\,\rn)$ convergence.
By an argument similar to that used in the proof of \cite[Proposition 3.25]{HMM11},
we can show Lemma \ref{lem tent-decompositon},
the details being omitted.
\begin{lem}[\cite{Ru07}]\label{lem tent-decompositon}
Let $p\in (0,1]$, $w\in A_\fz(\rn)$ and $f\in T^p(w,\,\rn)$.
Then there exist a sequence of $(w,\,p,\,2)$-atoms, $\{a_j\}_{j\in\nn}$,
and $\{\lz_j\}_{j\in\nn}\st\cc$ such that
\begin{equation}\label{eq 2.0x}
f=\sum_{j\in\nn}\lz_j a_j,
\end{equation}
where the series converges in $T^p(w,\,\rn)$. Moreover, there exist positive constants
$\wz C$ and $C$, independent of $f$, such that
\begin{equation*}
\wz C\|f\|_{T^p(w,\,\rn)}\le \lf\{\sum_{j\in\nn}|\lz_j|^p\r\}^{1/p}\le C\|f\|_{T^p(w,\,\rn)}.
\end{equation*}
Furthermore, if $f\in T^p(w,\,\rn)\cap T^2(w,\,\rn)$, then the series in \eqref{eq 2.0x}
converges in both $T^p(w,\,\rn)$ and $T^2(w,\,\rn)$.
\end{lem}

The following lemma establishes the complex interpolation property of the weighted tent spaces.
Noticing that, for any $w\in A_\fz(\rn)$,
$w(x)dx$ is a doubling measure on $\rn$, Lemma \ref{lem interpolation tent}
is just a special case of \cite[Proposition 3.18]{Am14}.
Here and hereafter, for any $\theta\in [0,\,1]$,
$[\cdot,\,\cdot]_\theta$ denotes the complex interpolation space
(see, for example, \cite[Chapter 4]{BL76} for the definition).

\begin{lem}[\cite{Am14}]\label{lem interpolation tent}
Let $w\in A_\fz(\rn)$. Then, for any $p_0,\,p_1\in [1,\,\fz)$ and $\theta\in[0,\,1]$,
it holds true that
\begin{equation*}
\lf[T^{p_0}(w,\,\rn),\,T^{p_1}(w,\,\rn)\r]_{\theta}=T^p(w,\,\rn),
\end{equation*}
where $1/p=\theta/p_0+(1-\theta)/p_1$.
\end{lem}

From \cite[Proposition 3.10]{Am14}, we deduce the following conclusion.
\begin{lem}[\cite{Am14}]\label{lem dual tent}
Let $p\in(1,\,\fz)$ and $w\in A_\fz(\rn)$. Then,
for any $f\in T^p(w,\,\rn)$ and $g\in T^{p'}(w,\,\rn)$, the pairing
\begin{equation*}
\langle f,\,g\rangle:=\iint_{\rr_+^{n+1}}f(y,t)\ov{g(y,t)}w(y)\,dy\,\frac{dt}{t}
\end{equation*}
realizes $T^{p'}(w,\,\rn)$ as the dual space of $T^p(w,\,\rn)$, up to equivalent
norms, where $1/p+1/p'=1$.
\end{lem}

\section{Square function characterizations of $H_{L_w}^p(\rn)$}\label{s3}
\hskip\parindent
In this section, we prove the square function characterizations of
$H_{L_w}^p(\rn)$ and we mainly follow the strategy used in \cite{HMM11}.
To this end, we first establish some technical lemmas.

Let $w\in A_2(\rn)$, $\oz:=\arctan(\Lambda/\lambda)$, $\mu\in(\omega,\,\pi/2)$
and $\psi\in\Psi(\Sigma_\mu^0)$. For any $f\in L^2(w,\,\rn)$
and $(x,t)\in\rr_+^{n+1}$, define
\begin{equation*}
Q_{\psi,L_w}(f)(x,t):=\psi(t^2L_w)(f)(x).
\end{equation*}
By \cite[Theorem F]{ADM96} and a simple calculation, we see that
$Q_{\psi,L_w}$ is bounded from $L^2(w,\,\rn)$ to $T^2(w,\,\rn)$.
For any $\psi\in\Psi(\Sigma_\mu^0)$, $F\in T^2(w,\,\rn)$ and $x\in\rn$, let
\begin{equation*}
\pi_{\psi,L_w}(F)(x):=\int_0^\fz \psi(t^2L_w)(F(\cdot,t))(x)\,\frac{dt}{t}.
\end{equation*}
Since $L^\ast_w$ is the adjoint operator of $L_w$ in $L^2(w,\,\rn)$,
we see that,
for any $f\in L^2(w,\,\rn)$ and $G\in T^2(w,\,\rn)$,
\begin{eqnarray}\label{eq 2-0}
\langle Q_{\psi, L^\ast_w}(f),\,G\rangle
&&=\int_0^\fz\int_\rn  Q_{\psi, L^\ast_w}(f)(x,t)\ov{G(x,t)}w(x)\,dx\,\frac{dt}{t}\\
&&=\int_0^\fz\int_\rn \psi(t^2L^\ast_w)(f)(x)\ov{G(x,t)}w(x)\,dx\,\frac{dt}{t}\noz\\
&&=\int_0^\fz\int_\rn f(x)\ov{\psi(t^2L_w)(G(\cdot,t))(x)}w(x)\,dx\,\frac{dt}{t}\noz\\
&&=\int_\rn f(x)\ov{\pi_{\psi,L_w}(G)(x)}w(x)\,dx=:(f,\,\pi_{\psi,L_w}(G))_w.\noz
\end{eqnarray}
This, together with the fact that
$Q_{\psi,L^\ast_w}$ is bounded from $L^2(w,\,\rn)$ to $T^2(w,\,\rn)$,
implies that $\pi_{\psi,L_w}$ is the adjoint operator of $Q_{\psi,L^\ast_w}$ and
bounded from $T^2(w,\,\rn)$ to $L^2(w,\,\rn)$.

For any $\psi,\,\wz\psi\in\Psi(\Sigma_\mu^0)$, $f\in H_\fz(\Sigma_\mu^0)$,
$F\in T^2(w,\,\rn)$ and $(x,t)\in\rr_+^n$,
define
\begin{eqnarray*}
Q^f(F)(x,t)
&&:=Q_{\psi,L_w}\circ f\circ \psi_{\wz\psi,L_w}(F)(x,t)\\
&&:=\int_0^\fz \psi(t^2L_w)\lf(f(L_w)\wz\psi(s^2L_w)(F(\cdot,s))\r)(x)\,\frac{ds}{s}.
\end{eqnarray*}
From the above argument and the fact that $L_w$ has a bounded holomorphic functional calculus,
it follows that $Q^f$ is bounded from $T^2(w,\,\rn)$ to itself.
Moreover, by Lemmas \ref{lem interpolation tent} and \ref{lem dual tent},
we have the following conclusion, which is an analogue of \cite[Proposition 4.4]{HMM11}.

\begin{lem}\label{lem Qfunction}
Let $w\in A_2(\rn)$, $\omega:=\arctan(\Lambda/\lambda)$ and $\mu\in(\omega,\,\pi/2)$.
Then, for any $\psi,\,\wz\psi\in\Psi(\Sigma_\mu^0)$ and $f\in H_\fz(\Sigma_\mu^0)$,
the operator $Q^f:=Q_{\psi,L_w}\circ f\circ \psi_{\wz\psi,L_w}$
is bounded from $T^p(w,\,\rn)$ to itself if

{\rm (i)} $p\in(0,\,2]$, $\psi\in\Psi_{\az,\,\bz}(\Sigma_\mu^0)$
and $\wz\psi\in\Psi_{\bz,\,\az}(\Sigma_\mu^0)$, or

{\rm (ii)} $p\in(2,\,\fz)$, $\psi\in\Psi_{\bz,\,\az}(\Sigma_\mu^0)$
and $\wz\psi\in\Psi_{\az,\,\bz}(\Sigma_\mu^0)$,\\
where $\az\in(0,\,\fz)$ and $\bz\in(n(\max\{\frac{1}{p},\,1\}-\frac{1}{2}),\,\fz)$.
\end{lem}

\begin{proof}
We first prove Lemma \ref{lem Qfunction} in case $p\in(0,\,1]$.
By Lemma \ref{lem 2.0}, we see that, for any $a\in(0,\,\az)$, $b\in(0,\,\beta)$,
$F\in T^p(w,\,\rn)$ and $(x,s)\in\mathbb{R}^{n+1}_+$,
\begin{eqnarray}\label{eq 2.xx}
Q^f(F)(x,s)=\int_0^\fz\min\lf\{\lf(\frac{s}{t}\r)^{2a},\,\lf(\frac{t}{s}\r)^{2b}\r\}
T_{s^2,t^2}(F(\cdot,t))(x)\,\frac{dt}{t},
\end{eqnarray}
where the family $\{T_{s,t}\}_{s,t>0}$ of sublinear operators has the following
properties:

(i) For any $t\in[s,\,\fz)$, any closed sets $E$ and $F$ of
$\rn$ and $g\in L^2(w,\,\rn)$ with $\supp g\st E$,
\begin{eqnarray}\label{eq 2.1}
\|T_{s,t}(g)\|_{L^2(w,\,F)}\ls \|f\|_{L^\fz(\Sigma_\mu^0)}
\lf[\min\lf\{1,\,\frac{t}{[d(E,\,F)]^2}\r\}\r]^{\beta+a}\|g\|_{L^2(w,\,E)}.
\end{eqnarray}

(ii) For any $s\in[t,\,\fz)$, any closed sets $E$ and $F$ of
$\rn$ and $g\in L^2(w,\,\rn)$ with $\supp g\st E$,
\begin{eqnarray}\label{eq 2.2}
\|T_{s,t}(g)\|_{L^2(w,\,F)}\ls \|f\|_{L^\fz(\Sigma_\mu^0)}
\lf[\min\lf\{1,\,\frac{s}{[d(E,\,F)]^2}\r\}\r]^{\az+b}\|g\|_{L^2(w,\,E)}.
\end{eqnarray}
Fix $b\in(n(\max\{\frac{1}{p},\,1\}-\frac12),\,\fz)$ and
choose a constant
$$M\in(n[\max\{{1}/{p},\,1\}-1/2],\,\min\{\az+b,\,\bz+a\}).$$
Then, from \eqref{eq 2.1} and \eqref{eq 2.2}, it follows that,
for any $s,\,t>0$, any closed subsets $E$ and $F$ of $\rn$
and $g\in L^2(w,\,\rn)$ with $\supp g\st E$,
\begin{eqnarray}\label{eq 2.4}
\|T_{s^2,t^2}(g)\|_{L^2(w,\,F)}\ls \|f\|_{L^\fz(\Sigma_\mu^0)}
\lf[\min\lf\{1,\,\frac{\max\{s^2,\,t^2\}}{[d(E,\,F)]^2}\r\}\r]^{M}\|g\|_{L^2(w,\,E)}.
\end{eqnarray}

It is easy to see that $T^2(w,\,\rn)\cap T^p(w,\,\rn)$ is dense in $T^p(w,\,\rn)$
(see the proof of \cite[Proposition 3.25]{HMM11}).
By this, we claim that,
to prove Lemma \ref{lem Qfunction} in case $p\in(0,\,1]$, it suffices to prove that,
for any $(w,\,p,\,2)$-atom $A$,
\begin{eqnarray}\label{eq 2.3}
\lf\|Q^f(A)\r\|_{T^p(w,\,\rn)}\ls 1.
\end{eqnarray}
Indeed, from Lemma \ref{lem tent-decompositon}, we deduce that,
for any $F\in T^2(w,\,\rn)\cap T^p(w,\,\rn)$, there exist a sequence of $(w,\,p,\,2)$-atoms, $\{A_j\}_{j\in\nn}$, and $\{\lz_j\}_{j\in\nn}\st\cc$ such that
\begin{eqnarray}\label{eq 2.3x}
F=\sum_{j\in\nn}\lz_j A_j\ \ \text{in}\ \ T^2(w,\,\rn)\cap T^p(w,\,\rn)
\end{eqnarray}
and
\begin{eqnarray}\label{eq 2.3y}
\lf[\sum_{j\in\nn}|\lz_j|^p\r]^{1/p}\sim \|F\|_{L^p(w,\,\rn)}.
\end{eqnarray}
For any $N\in\nn$, let $S_N:=\sum_{j=1}^N \lz_j A_j$.
By \eqref{eq 2.3x} and the fact that $Q^f$ is bounded on $T^2(w,\,\rn)$,
we know that there exists a subsequence of $\{S_N\}_{N\in\nn}$ (without loss of generality,
we use the same notation as the original sequence) such that, for almost every
$(y,\,t)\in \rr_+^{n+1}$, $\lim_{N\to\fz}Q^f(S_N)(y,\,t)=Q^f(F)(y,\,t)$.
From this, \eqref{eq 2.3} and \eqref{eq 2.3y}, it follows that
\begin{eqnarray*}
\lf\|Q^f(F)\r\|_{T^p(w,\,\rn)}
\le\lf[\sum_{j=1}^\fz |\lz_j|^p\lf\|Q^f(A_j)\r\|^p_{T^p(w,\,\rn)}\r]^{\frac1p}
\ls \lf[\sum_{j=1}^\fz |\lz_j|^p\r]^{\frac1p}
\sim \|F\|_{T^p(w,\,\rn)},
\end{eqnarray*}
which is the desired conclusion.

Next, we prove \eqref{eq 2.3}. For any $(w,\,p,\,2)$-atom $A$, there
exists a ball $B:=B(x_B,r_B)$ of $\rn$, with $x_B\in\rn$ and $r_B\in(0,\,\fz)$,
such that $\supp A\st \wh{B}$.
Let $S_1(\wh{B}):=\wh{2B}$ and
$S_j(\wh{B}):=\wh{2^j B}\setminus\wh{2^{j-1}B}$ for $j\in\nn\cap [2,\,\fz)$.
It is easy to see that $\mathbb{R}^{n+1}_+=\cup_{j=1}^\fz S_j(\wh{B})$.
Next, for $j\in\nn$, we estimate $\|\chi_{S_j(\wh{B})}Q^f(A)\|_{T^p(w,\,\rn)}$.

When $j=1$, it is easy to see that, for any $(y,\,t)\in\wh{2B}$,
$\{x\in\rn:\ |x-y|<t\}\st 3B$.
From this, the fact that $p\in(0,\,1]$,
the H\"{o}lder inequality, Lemma \ref{lem Ap-2},
the fact that $Q^f$ is bounded on $T^2(w,\,\rn)$
and \eqref{eq tent atom}, it follows that
\begin{eqnarray*}
\lf\|\chi_{\wh{2B}}Q^f(A)\r\|_{T^p(w,\,\rn)}
&&=\lf\{\int_{3B}\lf[\iint_{\Gamma(x)}\lf|\lf(\chi_{\wh{2B}}Q^f(A)\r)(y,t)\r|^2
\frac{w(y)\,dy\,dt}{w(B(x,t))\,t}\r]^{\frac{p}{2}}w(x)\,dx\r\}^{\frac{1}{p}}\\
&&\le \lf\|\chi_{\wh{2B}}Q^f(A)\r\|_{T^2(w,\,\rn)}\lf[w(3B)\r]^{\frac1p-\frac12}\\
&&\ls \|A\|_{T^2(w,\,\rn)}[w(B)]^{\frac1p-\frac12}\ls 1.
\end{eqnarray*}

When $j\in\nn\cap[2,\,\fz)$, it is easy to see that, for any $(y,\,t)\in S_j(\wh{B})$,
$\{x\in\rn:\ |x-y|<t\}\st 2^{j+1}B$. By this, the fact that $p\in(0,\,1]$,
the H\"{o}lder inequality, the Fubini theorem and Lemma \ref{lem Ap-2},
we conclude that
\begin{eqnarray}\label{eq 2.y1}
\qquad&&\lf\|\chi_{S_j(\wh{B})}Q^f(A)\r\|_{T^p(w,\,\rn)}\\
&&\hs=\lf\{\int_{2^{j+1}B}\lf[\iint_{\Gamma(x)}\lf|\lf(\chi_{S_j(\wh{B})}Q^f(A)\r)(y,s)\r|^2
\frac{w(y)\,dy\,ds}{w(B(x,s))\,s}\r]^{\frac{p}{2}}w(x)\,dx\r\}^{\frac{1}{p}}\noz\\
&&\hs\le\lf\{\int_{2^{j+1}B}\iint_{\Gamma(x)}\lf|\lf(\chi_{S_j(\wh{B})}Q^f(A)\r)(y,s)\r|^2
\frac{w(y)\,dy\,ds}{w(B(x,s))\,s}w(x)\,dx\r\}^{\frac{1}{2}}\lf[w(2^{j+1}B)\r]^{\frac1p-\frac12}\noz\\
&&\hs\ls\lf\{\lf[\int_0^{2^{j-1}r_B}\int_\rn
\lf|\lf(\chi_{S_j(\wh{B})}Q^f(A)\r)(y,s)\r|^2w(y)\,dy\,\frac{ds}{s}\r]^{\frac12}
+\lf[\int_{2^{j-1}r_B}^{2^jr_B}\int_\rn\cdots\r]^{\frac12}\r\}\noz\\
&&\hs\hs\times2^{2nj(\frac1p-\frac12)}[w(B)]^{\frac1p-\frac12}\noz\\
&&\hs =:\{{\rm I}+{\rm II}\}2^{2nj(\frac1p-\frac12)}[w(B)]^{\frac1p-\frac12}.\noz
\end{eqnarray}
For ${\rm II}$, from \eqref{eq 2.xx}, \eqref{eq 2.4},
the Minkowski inequality and the H\"{o}lder inequality, we deduce that
\begin{eqnarray}\label{eq 2.y2}
\qquad{\rm II}
&&=\lf[\int_{2^{j-1}r_B}^{2^jr_B}\int_\rn\chi_{S_j(\wh{B})}(x,s)
\lf|\int_0^{r_B}\lf(\frac{t}{s}\r)^{2b}T_{s^2,t^2}(A(\cdot,t))(x)\,\frac{dt}{t}\r|^2
w(x)\,dx\,\frac{ds}{s}\r]^{\frac12}\\
&&\le \int_0^{r_B}\lf(\frac{t}{2^jr_B}\r)^{2b}\lf[\int_{2^{j-1}r_B}^{2^jr_B}\int_\rn
\chi_{S_j(\wh{B})}(x,s)
\lf|T_{s^2,t^2}(A(\cdot,t))(x)\r|^2w(x)\,dx\,\frac{ds}{s}\r]^{\frac12}\,\frac{dt}{t}\noz\\
&&\ls\int_0^{r_B}\lf(\frac{t}{2^jr_B}\r)^{2b}\lf\|A(\cdot,t)\r\|_{L^2(w,\,B)}\,\frac{dt}{t}\noz\\
&&\ls\lf\{\int_0^{r_B}\lf(\frac{t}{2^jr_B}\r)^{4b}\,\frac{dt}{t}\r\}^{\frac12}
\lf\{\int_0^{r_B}\int_B\lf|A(y,t)\r|^2w(y)\,dy\,\frac{dt}{t}\r\}^{\frac12}\noz\\
&&\ls[w(B)]^{\frac12-\frac1p}2^{-2bj}.\noz
\end{eqnarray}
For ${\rm I}$, by \eqref{eq 2.xx} and the Minkowski inequality, we have
\begin{eqnarray}\label{eq 2.y3}
{\rm I}
&&=\lf[\int_0^{2^{j-1}r_B}\int_\rn\chi_{S_j(\wh{B})}(x,s)\r.\\
&&\hs\times\lf.\lf|\int_0^{r_B}\psi(s^2L_w)f(L_w)\wz\psi(t^2L_w)(A(\cdot,t))(x)\,\frac{dt}{t}\r|^2
w(x)\,dx\,\frac{ds}{s}\r]^{\frac12}\noz\\
&&\le \int_0^{r_B}\lf[\int_0^t\int_\rn\lf(\frac{s}{t}\r)^{4a}
\chi_{S_j(\wh{B})}(x,s)
\lf|T_{s^2,t^2}(A(\cdot,t))(x)\r|^2w(x)\,dx\,\frac{ds}{s}\r]^{\frac12}\,\frac{dt}{t}\noz\\
&&\hs+\int_0^{r_B}\lf[\int_t^{2^{j-1}r_B}\int_\rn\lf(\frac{t}{s}\r)^{4b}
\cdots \r]^{\frac12}\,\frac{dt}{t}\noz\\
&&=:{\rm I}_1+{\rm I}_2.\noz
\end{eqnarray}
For ${\rm I}_1$, by \eqref{eq 2.4}, the H\"{o}lder inequality and \eqref{eq tent atom},
we see that
\begin{eqnarray}\label{eq 2.y4}
{\rm I}_1
&&\ls\int_0^{r_B}\lf\{\int_0^t\lf(\frac{s}{t}\r)^{4a}\lf[\frac{t^2}{(2^jr_B)^2}\r]^{2M}
\|A(\cdot,t)\|_{L^2(w,\,B)}^2\,\frac{ds}{s}\r\}^{\frac12}\,\frac{dt}{t}\\
&&\sim\frac{1}{(2^jr_B)^{2M}}\int_0^{r_B}\|A(\cdot,t)\|_{L^2(w,\,B)}t^{2M}\,\frac{dt}{t}\noz\\
&&\ls\frac{1}{(2^jr_B)^{2M}}\lf[\int_0^{r_B}\int_B|A(y,t)|^2w(y)\,dy\,\frac{dt}{t}\r]^{\frac12}
\lf[\int_0^{r_B}t^{4M-1}\,dt\r]^{\frac12}\noz\\
&&\ls 2^{-2Mj}[w(B)]^{\frac12-\frac1p}.\noz
\end{eqnarray}
Similarly, we see that
\begin{eqnarray}\label{eq 2.y5}
{\rm I}_2
&&\ls \int_0^{r_B}\lf[\lf(\frac{t}{2^jr_B}\r)^{2b}+\lf(\frac{t}{2^jr_B}\r)^{2M}\r]
\|A(\cdot,t)\|_{L^2(w,\,B)}\,\frac{dt}{t}\\
&&\ls\lf(2^{-2bj}+2^{-2Mj}\r)[w(B)]^{\frac12-\frac1p}.\noz
\end{eqnarray}
From \eqref{eq 2.y1}, \eqref{eq 2.y2}, \eqref{eq 2.y3}, \eqref{eq 2.y4}, \eqref{eq 2.y5}
and the fact that $M>n(\max\{\frac1p,\,1\}-\frac12)$,
we deduce that, for any $(w,\,p,\,2)$-atom $A$,
\begin{eqnarray*}
\lf\|Q^f(A)\r\|_{T^p(w,\,\rn)}\ls 1,
\end{eqnarray*}
which completes the proof of Lemma \ref{lem Qfunction} in case $p\in(0,\,1]$.
Since we already known that, for any $F\in T^2(w,\,\rn)$,
$$\lf\|Q^f(F)\r\|_{T^2(w,\,\rn)}\ls\|F\|_{T^2(w,\,\rn)},$$
then, by Lemma \ref{lem interpolation tent} and the well-known property of
interpolation spaces (see, for example, \cite[Theorem 4.1.2]{BL76}),
we find that, for any $p\in(1,\,2]$ and $F\in T^p(w,\,\rn)$,
$$\lf\|Q^f(F)\r\|_{T^p(w,\,\rn)}\ls\|F\|_{T^p(w,\,\rn)}.$$
This proves Lemma \ref{lem Qfunction} in case $p\in(1,\,2]$.

By the above argument and the duality,
it is easy to prove Lemma \ref{lem Qfunction} in case $p\in(2,\,\fz)$.
This finishes the proof of Lemma \ref{lem Qfunction}.
\end{proof}

We have the following Calder\'{o}n reproducing formula.
\begin{lem}\label{lem c-z}
Let $w\in A_2(\rn)$, $\omega:=\arctan(\Lambda/\lambda)$ and $\mu\in(\omega,\,\pi/2)$.
For any $\psi,\,\wz\psi\in\Psi(\Sigma_\mu^0)$ satisfying
$\int_0^\fz\psi(t^2)\wz\psi(t^2)\,\frac{dt}{t}=1$, and any $f\in L^2(w,\,\rn)$,
it holds true that
\begin{eqnarray*}
\pi_{\psi,L_w}\circ Q_{\wz\psi,L_w}(f)=\pi_{\wz\psi,L_w}\circ Q_{\psi,L_w}(f)=f\ \ \
\text{in}\ L^2(w,\,\rn).
\end{eqnarray*}
\end{lem}
\begin{proof}
By a simple calculation, we see that, for any $z\in\Sigma_\mu^0$,
$$\int_0^\fz\psi(t^2z)\wz\psi(t^2z)\,\frac{dt}{t}
=\int_0^\fz\psi(t^2)\wz\psi(t^2)\,\frac{dt}{t}=1.$$
By \eqref{eq def1} and the properties of holomorphic functional calculi
(see \cite[Lecture 2]{ADM96}), we see that
\begin{eqnarray*}
\psi(t^2L_w)\wz\psi(t^2L_w)=\int_\gz\psi(t^2z)\wz\psi(t^2z)(zI-L_w)^{-1}\,dz,
\end{eqnarray*}
where $\gamma:=\{re^{i\nu}:\ r\in(0,\,\fz)\}\cup\{re^{-i\nu}:\ r\in(0,\,\fz)\}$
and $\nu\in(\omega,\,\mu)$.
Hence,
\begin{eqnarray*}
\pi_{\psi,L_w}\circ Q_{\wz\psi,L_w}
&&=\int_0^\fz\psi(t^2L_w)\wz\psi(t^2L_w)\,\frac{dt}{t}\\
&&=\int_0^\fz\int_\gz\psi(t^2z)\wz\psi(t^2z)(zI-L_w)^{-1}\,dz\,\frac{dt}{t}\\
&&=\int_\gz\int_0^\fz\psi(t^2z)\wz\psi(t^2z)\,\frac{dt}{t}(zI-L_w)^{-1}\,dz
=\int_\gz(zI-L_w)^{-1}\,dz=I.
\end{eqnarray*}
By changing the roles of $\psi$ and $\wz\psi$, we obtain
$\pi_{\wz\psi,L_w}\circ Q_{\psi,L_w}=I$.
This finishes the proof of Lemma \ref{lem c-z}.
\end{proof}

For $w\in A_2(\rn)$, $\omega:=\arctan(\Lambda/\lambda)$
and $\mu\in(\omega,\,\pi/2)$, let
\begin{eqnarray*}
K(\Sigma_\mu^0):=\{f\in H_\fz(\Sigma_\mu^0):\ f(t)\equiv0\ \text{for all}\ t\in(0,\,\fz)\}.
\end{eqnarray*}
\begin{rem}\label{rem 1}
For any $\psi\in\Psi(\Sigma_\mu^0)\setminus K(\Sigma_\mu^0)$
and $z\in\Sigma_\mu^0$, by taking
$$\wz\psi(z):=\frac{2\ov{\psi(z)}}{\int_0^\fz |\psi(t)|^2\,\frac{dt}{t}},$$
we find that
$$\int_0^\fz\psi(t^2)\wz\psi(t^2)\,\frac{dt}{t}
=\frac{2\int_0^\fz|\psi(t^2)|^2\,\frac{dt}{t}}
{\int_0^\fz|\psi(t)|^2\,\frac{dt}{t}}=1.$$
Thus, $\psi$ and $\wz\psi$ satisfy the assumptions of Lemma \ref{lem c-z}.
\end{rem}

\begin{defn}\label{def square Hardy1}
Let $w\in A_2(\rn)$, $\omega:=\arctan(\Lambda/\lambda)$ and $\mu\in(\omega,\,\pi/2)$. Let

(i) $p\in (0,\,2]$ and $\psi\in \Psi_{\az,\beta}(\Sigma_\mu^0)$, or

(ii) $p\in(2,\,\fz)$ and $\psi\in \Psi_{\bz,\az}(\Sigma_\mu^0)$,\\
where $\az\in(0,\,\fz)$ and $\bz\in(n[\max\{\frac{1}{p},\,1\}-\frac{1}{2}],\,\fz)$.
The Hardy space $H_{\psi,L_w}^p(\rn)$ is defined as the completion of the space
\begin{eqnarray*}
\{f\in L^2(w,\,\rn):\ Q_{\psi,L_w}(f)\in T^p(w,\,\rn)\}
\end{eqnarray*}
with respect to the (quasi-)norm
\begin{eqnarray*}
\|f\|_{H^p_{\psi,L_w}(\rn)}:=\|Q_{\psi,L_w}(f)\|_{T^p(w,\,\rn)}.
\end{eqnarray*}
\end{defn}

\begin{lem}\label{lem element}
Let $w\in A_2(\rn)$, $\omega:=\arctan(\Lambda/\lambda)$ and $\mu\in (\omega,\,\pi/2)$.

{\rm (i)} Let $p\in(0,\,2]$ and $\psi,\,\psi_0\in\Psi_{\az,\beta}(\Sigma_\mu^0)$, or
$p\in(2,\,\fz)$ and $\psi,\,\psi_0\in\Psi_{\beta,\az}(\Sigma_\mu^0)$,
where $\az\in(0,\,\fz)$, $\bz\in(n[\max\{\frac{1}{p},\,1\}-\frac{1}{2}],\,\fz)$
and $\psi,\,\psi_0\notin K(\Sigma_\mu^0)$.
Then there exists a positive constant $C$ such that, for any $f\in H_{\psi_0,L_w}^p(\rn)$,
\begin{eqnarray}\label{eq element1}
\|Q_{\psi,L_w}(f)\|_{T^p(w,\,\rn)}\le C\|f\|_{H_{\psi_0,L_w}^p(\rn)}.
\end{eqnarray}

{\rm (ii)} Let $p\in(0,\,2]$, $\psi\in\Psi_{\beta,\az}(\Sigma_\mu^0)$
and $\psi_0\in\Psi_{\az,\beta}(\Sigma_\mu^0)$, or
$p\in(2,\,\fz)$, $\psi\in\Psi_{\az,\beta}(\Sigma_\mu^0)$
and $\psi_0\in\Psi_{\beta,\az}(\Sigma_\mu^0)$,
where $\az\in(0,\,\fz)$, $\bz\in(n[\max\{\frac{1}{p},\,1\}-\frac{1}{2}],\,\fz)$
and $\psi,\,\psi_0\notin K(\Sigma_\mu^0)$.
Then there exists a positive constant $C$ such that, for any $f\in T^p(w,\,\rn)$,
\begin{eqnarray}\label{eq element2}
\|\pi_{\psi,L_w}(f)\|_{H^p_{\psi_0,L_w}(\rn)}\le C\|f\|_{T^p(w,\,\rn)}.
\end{eqnarray}
\end{lem}
\begin{proof}
We first prove {\rm (i)}.
If $p\in(0,\,2]$ and $\psi,\,\psi_0\in\Psi_{\az,\beta}(\Sigma_\mu^0)$,
choose a function $\wz\psi_0\in\Psi_{\beta,\az}(\Sigma_\mu^0)$ such that
$\int_0^\fz\wz\psi_0(t^2)\psi_0(t^2)\,\frac{dt}{t}=1$.
Hence, by Lemma \ref{lem c-z}, we know that, for any $f\in L^2(w,\,\rn)$,
$$f=\pi_{\wz\psi_0,L_w}\circ Q_{\psi_0,L_w}(f)\ \ \
\text{in}\ L^2(w,\,\rn).$$
From this and Lemma \ref{lem Qfunction}{\rm (i)},
it follows that, for any $f\in L^2(w,\,\rn)\cap H^p_{\psi_0,L_w}(\rn)$,
\begin{eqnarray*}
\|Q_{\psi,L_w}(f)\|_{T^p(w,\,\rn)}
&&\sim\|Q_{\psi,L_w}\circ\pi_{\wz\psi_0,L_w}\circ Q_{\psi_0,L_w}(f)\|_{T^p(w,\,\rn)}\\
&&\ls \|Q_{\psi_0,L_w}(f)\|_{T^p(w,\,\rn)}\sim \|f\|_{H^p_{\psi_0,L_w}(\rn)}.\noz
\end{eqnarray*}
Since $L^2(w,\,\rn)\cap H^p_{\psi_0,L_w}(\rn)$ is dense in $H^p_{\psi_0,L_w}(\rn)$,
we obtain \eqref{eq element1} in this case.
If $p\in(2,\,\fz)$ and $\psi,\,\psi_0\in\Psi_{\beta,\az}(\Sigma_\mu^0)$,
by Lemma \ref{lem Qfunction}{\rm (ii)} and an argument
similar to that used above, we find that
\eqref{eq element1} also holds true.

Next, we prove {\rm (ii)}.
If $p\in(0,\,2]$, $\psi\in\Psi_{\beta,\az}(\Sigma_\mu^0)$
and $\psi_0\in\Psi_{\az,\beta}(\Sigma_\mu^0)$, from
Lemma \ref{lem Qfunction}{\rm (i)}, we deduce that, for any $f\in T^p(w,\,\rn)\cap T^2(w,\,\rn)$,
\begin{eqnarray*}
\|\pi_{\psi,L_w}(f)\|_{H^p_{\psi_0,L_w}(\rn)}
=\|Q_{\psi_0,L_w}\circ\pi_{\psi,L_w}(f)\|_{T^p(w,\,\rn)}\ls\|f\|_{T^p(w,\,\rn)}.
\end{eqnarray*}
Since $T^p(w,\,\rn)\cap T^2(w,\,\rn)$ is dense in $T^p(w,\,\rn)$,
we obtain \eqref{eq element2} in this case.
If $p\in(2,\,\fz)$, $\psi\in\Psi_{\az,\beta}(\Sigma_\mu^0)$
and $\psi_0\in\Psi_{\beta,\az}(\Sigma_\mu^0)$, similarly,
by Lemma \ref{lem Qfunction}{\rm (ii)} and a density argument,
we find that \eqref{eq element2} also holds true in this case.
This finishes the proof of Lemma \ref{lem element}.
\end{proof}

\begin{prop}\label{pro 1}
Let $p\in(0,\,2]$, $w\in A_2(\rn)$, $\omega:=\arctan(\Lambda/\lambda)$ and $\mu\in (\omega,\,\pi/2)$.
For any $\psi\in\Psi_{\az,\beta}(\Sigma_\mu^0)\setminus K(\Sigma_\mu^0)$, where
$\az\in(0,\,\fz)$ and $\bz\in(n[\max\{\frac{1}{p},\,1\}-\frac{1}{2}],\,\fz)$,
$H^p_{L_w}(\rn)$ and $H^p_{\psi,L_w}(\rn)$ coincide with equivalent quasi-norms.
\end{prop}
\begin{proof}
Take $\psi_0(z):=ze^{-z}$ for any $z\in\Sigma_\mu^0$. It is easy to show that
$\psi_0\in\Psi_{\az,\beta}(\Sigma_\mu^0)$
with $\az\in(0,\,\fz)$ and $\bz\in(n[\max\{\frac{1}{p},\,1\}-\frac{1}{2}],\,\fz)$.
Then, from Definition \ref{def Hardy space}
and \eqref{eq element1}, we deduce that, for any $f\in H^p_{L_w}(\rn)\cap L^2(w,\,\rn)$,
\begin{eqnarray*}
\|f\|_{H^p_{\psi,L_w}(\rn)}
=\|Q_{\psi,L_w}(f)\|_{T^p(w,\,\rn)}\ls\|f\|_{H^p_{\psi_0,L_w}(\rn)}
\sim\|f\|_{H^p_{L_w}(\rn)},
\end{eqnarray*}
which implies that
\begin{eqnarray}\label{eq 2-1}
\lf[H^p_{L_w}(\rn)\cap L^2(w,\,\rn)\r]\st
\lf[H^p_{\psi,L_w}(\rn)\cap L^2(w,\,\rn)\r].
\end{eqnarray}

Next, we prove the reverse inclusion. To this end, we only need to show that,
for any $f\in H^p_{\psi,L_w}(\rn)\cap L^2(w,\,\rn)$,
\begin{eqnarray*}
\|f\|_{H^p_{L_w}(\rn)}\ls\|f\|_{H^p_{\psi,L_w}(\rn)}.
\end{eqnarray*}

Choose a function $\wz\psi\in \Psi_{\beta,\az}(\Sigma_\mu^0)$ such that
$\int_0^\fz\wz\psi(t^2)\psi(t^2)\,\frac{dt}{t}=1$.
Then, by Lemma \ref{lem c-z}, we know that, for any $f\in L^2(w,\,\rn)$,
$$\pi_{\wz\psi,L_w}\circ Q_{\psi,L_w}(f)=f
\ \ \ \text{in}\ L^2(w,\,\rn).$$
This, together with Lemma \ref{lem Qfunction}{\rm (i)}, implies that
\begin{eqnarray*}
\|f\|_{H^p_{L_w}(\rn)}&&=\|Q_{\psi_0,L_w}\|_{T^p(w,\,\rn)}
\sim\|Q_{\psi_0,L_w}\circ \pi_{\wz\psi,L_w}\circ Q_{\psi,L_w}(f)\|_{T^p(w,\,\rn)}\\
&&\ls \|Q_{\psi,L_w}(f)\|_{T^p(w,\,\rn)}\sim \|f\|_{H^p_{\psi,L_w}(\rn)}.
\end{eqnarray*}
Therefore,
$$\lf[H^p_{\psi,L_w}(\rn)\cap L^2(w,\,\rn)\r]
\st \lf[H^p_{L_w}(\rn)\cap L^2(w,\,\rn)\r].$$
This, together with \eqref{eq 2-1} and a density argument,
then finishes the proof of Proposition \ref{pro 1}.
\end{proof}

\begin{prop}\label{pro 2}
Let $p\in(2,\,\fz)$, $w\in A_2(\rn)$, $\omega:=\arctan(\Lambda/\lambda)$
and $\mu\in (\omega,\,\pi/2)$.
For any $\psi\in\Psi_{\beta,\az}(\Sigma_\mu^0)\setminus K(\Sigma_\mu^0)$, where
$\az\in(0,\,\fz)$ and $\bz\in(n[\max\{\frac{1}{p},\,1\}-\frac{1}{2}],\,\fz)$,
$H^p_{L_w}(\rn)$ and $H^p_{\psi,L_w}(\rn)$ coincide with equivalent norms.
\end{prop}
\begin{proof}
We first prove the following inclusion:
\begin{eqnarray}\label{eq 2-2}
H^p_{\psi,L_w}(\rn)\st\lf(H^{p'}_{L^\ast_w}(\rn)\r)^\ast= H^p_{L_w}(\rn).
\end{eqnarray}

Take a function $\wz\psi\in\Psi_{\az,\beta}(\Sigma_\mu^0)$ such that
$\int_0^\fz\psi(t^2)\wz\psi(t^2)\,\frac{dt}{t}=1$.
Then, for any $f\in L^2(w,\,\rn)\cap H^p_{\psi,L_w}(\rn)$
and $g\in L^2(w,\,\rn)\cap H^{p'}_{L^\ast_w}(\rn)$,
by Lemma \ref{lem c-z} and \eqref{eq 2-0}, we conclude that
\begin{eqnarray*}
\int_\rn f(x)\ov{g(x)}w(x)\,dx
&&=\int_\rn\pi_{\wz\psi,L_w}\circ Q_{\psi,L_w}(f)(x)\ov{g(x)}w(x)\,dx\\
&&=\int_0^\fz\int_\rn Q_{\psi,L_w}(f)(x,t)\ov{Q_{\wz\psi,L^\ast_w}(g)(x,t)}
w(x)\,dx\,\frac{dt}{t}.
\end{eqnarray*}
Hence, from this, Lemmas \ref{lem dual tent} and \ref{lem element}, and Proposition \ref{pro 1},
it follows that
\begin{eqnarray*}
\lf|\int_\rn f(x)\ov{g(x)}w(x)\,dx\r|
&&\ls\|Q_{\psi,L_w}(f)\|_{T^p(w,\,\rn)}
\|Q_{\wz\psi,L^\ast_w}(g)\|_{T^{p'}(w,\,\rn)}\\
&&\sim\|f\|_{H^p_{\psi,L_w}(\rn)}\|g\|_{H^{p'}_{L^\ast_w}(\rn)}.
\end{eqnarray*}
This, together with a density argument, implies \eqref{eq 2-2}.

Next, we prove the reverse inclusion of \eqref{eq 2-2}.
Take a function $\psi\in\Psi_{\beta,\az}(\Sigma_\mu^0)$ with
$\az\in(0,\,\fz)$ and $\bz\in(n[\max\{\frac{1}{p},1\}-\frac{1}{2}],\,\fz$).
Then, for any $F\in T^{p'}(w,\,\rn)$, by Lemma \ref{lem element} and
Proposition \ref{pro 1}, we obtain
$\|\pi_{\psi,L^\ast_w}(F)\|_{H^{p'}_{L_w^\ast}(\rn)}\ls\|F\|_{T^{p'}(w,\,\rn)}$,
which implies $\pi_{\psi,L^\ast_w}(F)\in H^{p'}_{L_w^\ast}(\rn)$.
For any $l\in(H^{p'}_{L_w^\ast}(\rn))^\ast$
and $F\in T^{p'}(w,\,\rn)$, define
\begin{eqnarray*}
\langle l,\,F\rangle:=l\lf(\pi_{\psi,L^\ast_w}(F)\r).
\end{eqnarray*}
Then we find that, for any $F\in T^{p'}(w,\,\rn)$,
\begin{eqnarray*}
|\langle l,\, F \rangle|\le \|l\|_{(H^{p'}_{L_w^\ast}(\rn))^\ast}
\|\pi_{\psi,L^\ast_w}F\|_{H^{p'}_{L_w^\ast}(\rn)}
\ls \|l\|_{(H^{p'}_{L_w^\ast}(\rn))^\ast}\|F\|_{T^{p'}(w,\,\rn)}.
\end{eqnarray*}
This implies $l\in (T^{p'}(w,\,\rn))^\ast$.
Since $(T^{p'}(w,\,\rn))^\ast= T^p(w,\,\rn)$ (see Lemma \ref{lem dual tent}),
we know that there exists a function $G_l\in T^p(w,\,\rn)$ such that,
for any $F\in T^{p'}(w,\,\rn)$,
\begin{eqnarray}\label{eq 2-3}
\langle l,\,F \rangle
= l\lf(\pi_{\psi,L^\ast_w}(F)\r)
=\int_0^\fz\int_\rn F(x,t)\ov{G_l(x,t)}w(x)\,\frac{dx\,dt}{t}
=\langle F,\,G_l \rangle
\end{eqnarray}
and
\begin{eqnarray}\label{eq 2-6}
\|G_l\|_{T^p(w,\,\rn)}\sim \|l\|_{\lf(H^{p'}_{L_w^\ast}(\rn)\r)^\ast}.
\end{eqnarray}
Choose a function $\wz\psi\in\Psi_{\az,\beta}(\Sigma_\mu^0)$
with $\az\in(0,\,\fz)$ and $\bz\in(n[\max\{\frac{1}{p},1\}-\frac{1}{2}],\,\fz)$
such that $\int_0^\fz \psi(t^2)\wz\psi(t^2)\,\frac{dt}{t}=1$.
By Lemma \ref{lem element}, we have
\begin{eqnarray}\label{eq 2-5}
\|\pi_{\wz\psi,L_w}(G_l)\|_{H_{\psi,L_w}^p(\rn)}\ls\|G_l\|_{T^p(w,\,\rn)}.
\end{eqnarray}

Next, we prove that, for any $f\in H^{p'}_{L^\ast_w}(\rn)$,
\begin{eqnarray}\label{eq 2-4}
\int_\rn f(x)\ov{\pi_{\wz\psi,L_w}(G_l)(x)}w(x)\,dx= l(f)
\end{eqnarray}
and
\begin{eqnarray}\label{eq 2-x}
\|\pi_{\wz\psi,L_w}(G_l)\|_{H_{\psi,L_w}^p(\rn)}
\sim \|l\|_{(H^{p'}_{L_w^\ast}(\rn))^\ast}.
\end{eqnarray}

Since $T^2(w,\,\rn)\cap T^p(w,\,\rn)$ is dense in $T^p(w,\,\rn)$,
by \eqref{eq 2-0} and a density argument,
we conclude that, for any $f\in L^2(w,\,\rn)\cap H_{L^\ast_w}^{p'}(\rn)$,
\begin{eqnarray*}
\int_\rn f(x)\ov{\pi_{\wz\psi,L_w}(G_l)(x)}w(x)\,dx
&&=\int_0^\fz\int_\rn Q_{\wz\psi,L_w^\ast}(f)(x,t)\ov{G_l(x,t)}w(x)\,dx\,\frac{dt}{t}\\
&&=\lf\langle Q_{\wz\psi,L_w^\ast}(f),\, G_l\r\rangle
\end{eqnarray*}
which, together with \eqref{eq 2-3} and Lemma \ref{lem c-z}, implies that,
for any $f\in L^2(w,\,\rn)\cap H_{L^\ast_w}^{p'}(\rn)$,
\begin{eqnarray*}
\int_\rn f(x)\ov{\pi_{\wz\psi,L_w}(G_l)(x)}w(x)\,dx
=l\lf(\pi_{\psi,L_w}\circ Q_{\wz\psi,L^\ast_w}(f)\r)
=l(f).
\end{eqnarray*}
By the fact that $L^2(w,\,\rn)\cap H_{L_w^\ast}^{p'}(\rn)$ is dense
in $H_{L_w^\ast}^{p'}(\rn)$, we obtain \eqref{eq 2-4}.
From \eqref{eq 2-6} and \eqref{eq 2-5}, it follows that
$\|\pi_{\wz\psi,L_w}(G_l)\|_{H_{\psi,L_w}^p(\rn)}\ls
\|l\|_{(H^{p'}_{L_w^\ast}(\rn))^\ast}$.
By \eqref{eq 2-4}, we further see that
$\|l\|_{(H^{p'}_{L_w^\ast}(\rn))^\ast}\ls\|\pi_{\wz\psi,L_w}(G_l)\|_{H_{\psi,L_w}^p(\rn)}$.
Therefore, \eqref{eq 2-x} holds true. This implies that
\begin{eqnarray*}
H^p_{L_w}(\rn)=\lf(H^p_{L^\ast_w}(\rn)\r)^\ast\st H^p_{\psi,L_w}(\rn),
\end{eqnarray*}
which, together with \eqref{eq 2-2}, then completes the proof of Proposition \ref{pro 2}.
\end{proof}
\section{Proof of Proposition \ref{thm equi}}\label{s3.x}
\hskip\parindent
In this section, we show Proposition \ref{thm equi}.

Let $w\in A_2(\rn)$. For any $k\in\nn$, $f\in L^2(w,\,\rn)$ and $x\in\rn$,
define
\begin{eqnarray*}
S_{L_w,k}(f)(x):=\lf[\iint_{\bgz(x)}\lf|(t^2L_w)^ke^{-t^2L_w}(f)(y)\r|^2w(y)\,
\frac{dy}{w(B(x,t))}\,\frac{dt}{t}\r]^{\frac{1}{2}}.
\end{eqnarray*}
Noticing that, for any $w\in A_2(\rn)$, $w(x)\,dx$ is a doubling measure on $\rn$,
by Proposition \ref{pro ODEB} and \cite[Theorem 2.13]{BCKYY13-1}, we know that,
for any given $p\in(\frac{2n}{n+1},\,\frac{2n}{n-1})$, there exists a positive
constant $C$ such that, for any $f\in L^p(w,\,\rn)$,
\begin{eqnarray}\label{eq Spb}
\|S_{L_w,k}(f)\|_{L^p(w,\,\rn)}\le C\|f\|_{L^p(w,\,\rn)}.
\end{eqnarray}
\begin{proof}[Proof of Proposition \ref{thm equi}]
For $p\in(\frac{2n}{n+1},\,\frac{2n}{n-1})$, by Propositions \ref{pro 1}
and \ref{pro 2}, we see that $L^2(w,\,\rn)\cap H_{L_w}^p(\rn)$ is dense in
$H_{L_w}^p(\rn)$. Since $L^2(w,\,\rn)\cap L^p(w,\,\rn)$ is dense in $L^p(w,\,\rn)$,
to prove Proposition \ref{thm equi}, we only need to show that
\begin{eqnarray}\label{eq 2-x1}
\lf[L^2(w,\,\rn)\cap H_{L_w}^p(\rn)\r]=\lf[L^2(w,\,\rn)\cap L^p(w,\,\rn)\r]
\end{eqnarray}
and, for any $f\in L^2(w,\,\rn)\cap L^p(w,\,\rn)$,
\begin{eqnarray}\label{eq 2-x2}
\|f\|_{H_{L_w}^p(\rn)}\sim\|f\|_{L^p(w,\,\rn)}.
\end{eqnarray}

For any $k\in\nn$ with $k>n(\max\{\frac{1}{p},\,1\}-\frac12)$,
take $\psi(z):=z^ke^{-z}$ for all $z\in\Sigma_\mu^0$.
From \eqref{eq 2.0}, it is easy to see that, for any $\az\in(0,\,\fz)$,
$\psi\in\Psi_{k,\az}(\Sigma_\mu^0)$.
Then, for any $f\in L^2(w,\,\rn)\cap L^p(w,\,\rn)$, we have
\begin{eqnarray}\label{eq 2-7}
\|f\|_{H_{L_w}^p(\rn)}&&\sim\lf\|\lf\{
\iint_{\Gamma(\cdot)}\lf|\psi(t^2L_w)(f)(y)\r|^2w(y)\,\frac{dy}{w(B(\cdot,t))}\,
\frac{dt}{t}\r\}^{\frac12}\r\|_{L^p(w,\,\rn)}\\
&&\sim \lf\|S_{L_w,k}(f)\r\|_{L^p(w,\,\rn)}
\ls\|f\|_{L^p(w,\,\rn)}.\noz
\end{eqnarray}

On the other hand, taking an appropriate $\wz\psi\in\Psi(\Sigma_\mu^0)$
and using \eqref{eq 2-0},
Lemmas \ref{lem c-z}, \ref{lem dual tent} and \ref{lem element},
Propositions \ref{pro 1} and \ref{pro 2}, we conclude that, for any
$f\in L^2(w,\,\rn)\cap H_{L_w}^p(\rn)$ and $g\in L^2(w,\,\rn)\cap L^p(w,\,\rn)$ with
$\|g\|_{L^{p'}(w,\,\rn)}=1$,
\begin{eqnarray*}
\lf|\int_\rn f(x)\ov{g(x)}w(x)\,dx\r|
&&\sim\lf|\int_\rn \pi_{\psi,L_w}\circ Q_{\wz\psi,L_w}(f)(x)\ov{g(x)}w(x)\,dx\r|\\
&&\ls\lf\|Q_{\wz\psi,L_w}(f)\r\|_{T^p(w,\,\rn)}
\lf\|Q_{\wz\psi,L_w^\ast}(g)\r\|_{T^{p'}(w,\,\rn)}\\
&&\ls\|f\|_{H_{L_w}^p(\rn)}\lf\|S_{L_w,k}(g)\r\|_{L^{p'}(w,\,\rn)}\\
&&\ls\|f\|_{H_{L_w}^p(\rn)}\|g\|_{L^{p'}(w,\,\rn)}\ls\|f\|_{H_{L_w}^p(\rn)}.
\end{eqnarray*}
This implies $\|f\|_{L^p(w,\,\rn)}\ls\|f\|_{H_{L_w}^p(\rn)}$.
By this and \eqref{eq 2-7}, we obtain \eqref{eq 2-x1} and \eqref{eq 2-x2},
which then completes the proof of Proposition \ref{thm equi}.
\end{proof}

\section{Proof of Proposition \ref{thm main1}}\label{s4}
\hskip\parindent
In this section, we show Proposition \ref{thm main1}. To this end,
we first establish some technical lemmas.

\begin{lem}\label{lem inter-Hardy}
Let $w\in A_2(\rn)$. For any $\theta\in (0,\,1)$ and $p_0,\,p_1\in [1,\,\fz)$,
\begin{eqnarray*}
\lf[H_{L_w}^{p_0}(\rn),\,H_{L_w}^{p_1}(\rn)\r]_{\theta}= H_{L_w}^p(\rn),
\end{eqnarray*}
where $1/p=(1-\theta)/{p_0}+\theta/{p_1}$.
\end{lem}
\begin{proof}
It is easy to see that $H_{L_w}^{p_0}(\rn)$ and $H_{L_w}^{p_1}(\rn)$
are \emph{compatible}, namely, $H_{L_w}^{p_0}(\rn)$ and $H_{L_w}^{p_1}(\rn)$
are subspaces of $H_{L_w}^{p_0}(\rn)+H_{L_w}^{p_1}(\rn)$ endowed with the norm
\begin{eqnarray*}
&&\|f\|_{H_{L_w}^{p_0}(\rn)+H_{L_w}^{p_1}(\rn)}\\
&&\hs:=\inf\lf\{\|f_0\|_{H_{L_w}^{p_0}(\rn)}
+\|f_1\|_{H_{L_w}^{p_1}(\rn)}:\ f_0+f_1=f,\ f_0\in H_{L_w}^{p_0}(\rn),\ f_1\in H_{L_w}^{p_1}(\rn)\r\},
\end{eqnarray*}
and
$H_{L_w}^{p_0}(\rn)\cap H_{L_w}^{p_1}(\rn)$ is dense in $H_{L_w}^{p_i}(\rn)$, $i\in\{0,\,1\}$.

By Lemmas \ref{lem element} and \ref{lem c-z}, we know that, for $i\in\{0,\,1\}$ and
suitable $\psi,\,\wz\psi\in\Psi(\Sigma_\mu^0)$,
$Q_{\psi,L_w}$ is bounded from
$H_{L_w}^{p_i}(\rn)$ to $T^{p_i}(w,\,\rn)$,
$\pi_{\wz\psi,L_w}$ is bounded from $T^{p_i}(w,\,\rn)$ to $H_{L_w}^{p_i}(\rn)$,
and $\pi_{\wz\psi,L_w}\circ Q_{\psi,L_w}=I$ on $H_{L_w}^{p_i}(\rn)$.
This implies that $\{H_{L_w}^{p_0}(\rn),\,H_{L_w}^{p_1}(\rn)\}$
is a \emph{retract} of $\{T^{p_0}(w,\,\rn),\,T^{p_1}(w,\,\rn)\}$
(see \cite[p.\,151]{Kal07} for the definition).
Since $H_{L_w}^{p_i}(\rn)$, $i\in\{0,\,1\}$, is a Banach space, we know that
$H_{L_w}^{p_i}(\rn)$, $i\in\{0,\,1\}$, is \emph{analytic convex}
(see \cite[p.\,145]{Kal07} for the definition).

Hence, from the above argument, \cite[Lemma 7.11]{Kal07}
and Lemma \ref{lem interpolation tent},
it follows that, for some suitable $\wz\psi\in\Psi(\Sigma_\mu^0)$, any $\theta\in(0,\,1)$
and $1/p=(1-\theta)/{p_0}+\theta/{p_1}$,
\begin{eqnarray*}
\lf[H_{L_w}^{p_0}(\rn),\,H_{L_w}^{p_1}(\rn)\r]_\theta
&&=\pi_{\wz\psi,L_w}\lf(\lf[T^{p_0}(w,\,\rn),\,T^{p_1}(w,\,\rn)\r]_\theta\r)\\
&&=\pi_{\wz\psi,L_w}\lf(T^p(w,\,\rn)\r)=H_{L_w}^p(\rn).
\end{eqnarray*}
This finishes the proof of Lemma \ref{lem inter-Hardy}.
\end{proof}

The following lemma is an analogue of \cite[Lemma 2.10]{CR13}.
\begin{lem}\label{lem 5.1}
Let $E$ and $F$ be two closed sets of $\rn$. Then
there exist positive constants $C$ and $c$ such that, for any $t\in(0,\,\fz)$
and $f\in L^2(w,\,\rn)$ with $\supp f\st E$,
\begin{eqnarray}\label{eq 5.0}
\|e^{-tL_w}(f)\|_{L^2(w,\,F)}\le Ce^{-\frac{d(E,\,F)^2}{ct}}\|f\|_{L^2(w,\,E)},
\end{eqnarray}
\begin{eqnarray}\label{eq 5.1}
\|\sqrt{t}\nabla e^{-tL_w}(f)\|_{L^2(w,\,F)}\le Ce^{-\frac{d(E,\,F)^2}{ct}}\|f\|_{L^2(w,\,E)}
\end{eqnarray}
and, for any $t\in(0,\,\fz)$ and $\vec{f}:=(f_1,\,\ldots,\,f_n)$ with
$f_i\in L^2(w,\,\rn)$, $\supp f_i\st E$, $i\in\{1,\,\ldots,\,n\}$,
\begin{eqnarray}\label{eq 5.2}
\lf\|\sqrt{t}e^{-tL_w}\lf(\frac{1}{w}\div \lf(w\vec{f}\r)\r)\r\|_{L^2(w,\,F)}
\le  Ce^{-\frac{d(E,\,F)^2}{ct}}\|\vec{f}\|_{L^2(w,\,E)}.
\end{eqnarray}
\end{lem}
\begin{proof}
Noticing \eqref{eq 5.0} and \eqref{eq 5.1} have been proved, respectively, in
\cite[Theorem 1.6]{CR08} and \cite[Proposition 2.7]{ZCJY14},
to prove Lemma \ref{lem 5.1}, we only need to show \eqref{eq 5.2}.

Indeed, for any $g\in L^2(w,\,F)$ with $\supp g\st F$ and $\|g\|_{L^2(w,\,F)}=1$,
by the H\"{o}lder inequality and \eqref{eq 5.1}, we have
\begin{eqnarray*}
&&\lf|\int_F\sqrt{t}e^{-tL_w}\lf(\frac{1}{w}\div(w\vec{f})\r)(x) g(x)w(x)\,dx\r|\\
&&\hs=\lf|\int_\rn\sqrt{t}\frac{1}{w(x)}\div(w\vec{f})(x)e^{-tL_w^\ast}(g)(x)w(x)\,dx\r|\\
&&\hs=\lf|\int_\rn\sqrt{t}w(x)\vec{f}(x)\cdot\nabla e^{-tL_w^\ast}(g)(x)\,dx\r|
\le\int_E|\vec{f}(x)|\lf|\sqrt{t}\nabla e^{-tL_w^\ast}(g)(x)\r|w(x)\,dx\\
&&\hs\le\lf[\int_E\lf|\sqrt{t}\nabla e^{-tL_w^\ast}(g)(x)\r|^2w(x)\,dx\r]^{\frac12}
\lf[\int_E |\vec{f}(x)|^2w(x)\,dx\r]^{\frac12}
\ls e^{-\frac{d(E,\,F)^2}{ct}}\|f\|_{L^2(w,\,E)},
\end{eqnarray*}
which, together with a dual argument, further implies that \eqref{eq 5.2}.
This finishes the proof of Lemma \ref{lem 5.1}.
\end{proof}

By Lemma \ref{lem 5.1}, we obtain the following lemma.
\begin{lem}\label{lem 5.2}
Let $m\in\nn$ and $E,\,F$ be closed sets of $\rn$.
Then there exist positive constants $C$ and $c$ such that, for any $t\in(0,\,\fz)$
and $\vec{f}=(f_1,\,\ldots,\,f_n)$, with $f_i\in L^2(w,\,\rn)$, $\supp f_i\st E$,
$i\in\{1,\,\ldots,\,n\}$,
\begin{eqnarray*}
\lf\|\sqrt{t}\nabla L_w^{-1/2}\lf(I-e^{-tL_w}\r)^m\lf(\frac{1}{w}\div (w\vec{f})\r)\r\|
_{L^2(w,\,F)}
\le C\lf(\frac{[d(E,\,F)]^2}{t}\r)^{-\lf(m+\frac12\r)}\|\vec{f}\|_{L^2(w,\,E)}
\end{eqnarray*}
and
\begin{eqnarray*}
\lf\|\sqrt{t}\nabla\lf(\nabla L_w^{-1/2}\lf(I-e^{-tL_w}\r)^m\r)^\ast(\vec{f})\r\|
_{L^2(w,\,F)}
\le C\lf(\frac{[d(E,\,F)]^2}{t}\r)^{-\lf(m+\frac12\r)}\|\vec{f}\|_{L^2(w,\,E)}.
\end{eqnarray*}
\end{lem}
The proof of Lemma \ref{lem 5.2} is a complete analogue of that of \cite[Lemma 2.2]{HM03},
the details being omitted.

The following local weighted Poincar\'{e} inequality is just \cite[Theorem (1.2)]{FKS82}.
\begin{lem}[\cite{FKS82}]\label{lem poincare}
Let $n\geq 2$.
For any given $p\in (1,\,\fz)$ and $w\in A_p(\rn)$, there exist positive constants
$C$ and $\delta$ such that, for any ball $B\equiv B(x_B,r_B)$ of $\rn$
with $x_B\in\rn$ and $r_B\in (0,\,\fz)$, any Lipschitz continuous function $u$
on $\bar{B}$, and any number $k\in [1,\,\frac{n}{n-1}+\delta)$,
\begin{equation*}
\lf[\frac{1}{w(B)}\int_B\lf|u(x)-u_B\r|^{kp}w(x)\,dx\r]^{\frac1{kp}}\le Cr_B
\lf[\frac{1}{w(B)}\int_B|\nab u(x)|^{p}w(x)\,dx\r]^{\frac1p},
\end{equation*}
where
\begin{eqnarray}\label{eq mean}
u_B:=\frac{1}{w(B)}\int_B u(x)w(x)\,dx.
\end{eqnarray}
\end{lem}

Let $w\in A_\fz(\rn)$. For any $p\in(0,\,\fz)$, the
\emph{weighted weak-$L^p$ space $L^{p,\fz}(w,\,\rn)$}
is defined as the set of all measurable functions $f$ on $\rn$ such that
\begin{eqnarray*}
\|f\|_{L^{p,\fz}(w,\,\rn)}
:=\sup_{\az\in(0,\,\fz)} \az\lf[w(\{x\in\rn:\ |f(x)|>\az\})\r]^{\frac1p}<\fz.
\end{eqnarray*}
For $L^{p,\fz}(w,\,\rn)$, we have the following Fatou lemma
(see \cite[Exercise 1.1.12]{Gra14} for its proof).
\begin{lem}\label{lem fatou}
Let $w\in A_\fz(\rn)$ and $p\in(0,\,\fz)$. Then there exists a positive
constant $C_{(w,\,p)}$, depending on $w$ and $p$, such that, for all
measurable functions $\{g_k\}_{k=1}^\fz$ on $\rn$,
\begin{eqnarray*}
\lf\|\liminf_{k\to\fz}|g_k|\r\|_{L^{p,\fz}(w,\,\rn)}\le C_{(w,\,p)}
\liminf_{k\to\fz}\|g_k\|_{L^{p,\fz}(w,\,\rn)}.
\end{eqnarray*}
\end{lem}

By Lemmas \ref{lem 5.1}, \ref{lem 5.2} and \ref{lem poincare}, we obtain the
following theorem which establishes the boundedness of Riesz transform
$\nabla L^{-1/2}_w$ on $L^q(w,\,\rn)$.
Theorem \ref{pro 4} is an analogue of \cite[Theorem 1.2]{HM03}.

\begin{thm}\label{pro 4}
Let $p:=\frac{2n}{n+1}$ and $w\in A_2(\rn)$.
Then there exists a positive constant $C$ such that,
for any $\az\in(0,\,\fz)$ and $f\in L^p(w,\,\rn)$,
\begin{eqnarray}\label{eq pro 4}
w\lf(\lf\{x\in\rn:\ \lf|\nabla L^{-1/2}_w(f)(x)\r|>\az\r\}\r)\le
\frac{C}{\az^p}\int_\rn|f(x)|^pw(x)\,dx.
\end{eqnarray}
Moreover, for any given $q\in(p,\,2]$, there exists a positive
constant $C$ such that, for any $f\in L^q(w,\,\rn)$,
\begin{eqnarray*}
\lf\|\nabla L_w^{-1/2}(f)\r\|_{L^q(w,\,\rn)}\le C\|f\|_{L^q(w,\,\rn)}.
\end{eqnarray*}
\end{thm}
\begin{proof}
To prove Theorem \ref{pro 4}, by the Marcinkiewicz interpolation theorem
and a density argument,
we only need to show, for any $f\in L^2(w,\,\rn)\cap L^p(w,\,\rn)$,
\eqref{eq pro 4} holds true.
Indeed, since $L^2(w,\,\rn)\cap L^p(w,\,\rn)$ is dense in $L^p(w,\,\rn)$,
we know that, for any $f\in L^p(w,\,\rn)$, there exists a family of functions,
$\{f_k\}_{k\in\nn}\st L^2(w,\,\rn)\cap L^p(w,\,\rn)$, such that
$\lim_{k\to\fz}\|f_k-f\|_{L^p(w,\,\rn)}=0$.
Thus, $\{f_k\}_{k\in\nn}$ is a Cauchy sequence in $L^p(w,\,\rn)$.
By \eqref{eq pro 4}, we know that $\{\nabla L_w^{-1/2}(f_k)\}_{k\in\nn}$ is
a Cauchy sequence in measure $w(x)\,dx$. From \cite[Theorem 1.1.13]{Gra14},
it follows that there exists a subsequence of $\{\nabla L_w^{-1/2}(f_k)\}_{k\in\nn}$
(without loss of generality, we may use the same notation as the original sequence)
such that, for almost every $x\in\rn$, $\lim_{k\to\fz}\nabla L_w^{-1/2}(f_k)(x)$ exists.
For almost every $x\in\rn$, let
$$\nabla L_w^{-1/2}(f)(x):=\lim_{k\to\fz}\nabla L_w^{-1/2}(f_k)(x).$$
It is easy to see that $\nabla L_w^{-1/2}(f)$ is well defined.
By Lemma \ref{lem fatou}, we further find that,
for any $f\in L^p(w,\,\rn)$, \eqref{eq pro 4} holds true.

Next, we prove that, for any $f\in L^2(w,\,\rn)\cap L^p(w,\,\rn)$, \eqref{eq pro 4}
holds true. To this end,
for any $f\in L^1_{\rm loc}(w,\,\rn)$ and $x\in\rn$,
let
\begin{eqnarray}\label{eq hl}
M_w(f)(x):=\sup_{B\ni x}\frac{1}{w(B)}\int_B |f(y)|w(y)\,dy,
\end{eqnarray}
where the supremum is taken over all balls containing $x$.
For any $f\in L^2(w,\,\rn)\cap L^p(w,\,\rn)$,
by the generalized Calder\'{o}n-Zygmund decomposition
\cite[p.\,17, Theorem 2]{S93} and its proof, we know that
there exist positive constants $C$ and $N$ such that, for any $\az\in(0,\,\fz)$,
there exist a collection of balls, $\{B_k\}_{k=1}^\fz:=\{B(x_k,\, r_k)\}_{k=1}^\fz$ of $\rn$
with $x_B\in\rn$ and $r_B\in(0,\,\fz)$,
a family $\{b_k\}_{k=1}^\fz$ of functions and
an almost everywhere bounded function $g$ such that the following properties hold true:
\begin{eqnarray}\label{eq 4.2x}
f(x)=g(x)+\sum_{k=1}^\fz b_k(x)\ \ \text{for almost every}\ x\in\rn;
\end{eqnarray}
\begin{eqnarray}\label{eq 4.2}
|g(x)|\le C\az\ \ \text{for almost every}\ x\in\rn;
\end{eqnarray}
\begin{eqnarray}\label{eq 4.3}
\qquad\supp b_k\st B_k,\ \ \int_{B_k}b_k(x)w(x)\,dx=0\ \ and\
\lf[\frac{1}{w(B_k)}\int_{B_k}|b_k(x)|^pw(x)\,dx\r]^{\frac1p}\le C\az;
\end{eqnarray}
\begin{eqnarray}\label{eq 4.4}
\sum_{k=1}^\fz w(B_k)\le C\az^{-p}\int_\rn|f(x)|^pw(x)\,dx;
\end{eqnarray}
\begin{eqnarray}\label{eq 4.5}
\sum_{k=1}^\fz \chi_{B_k}(x)\le N.
\end{eqnarray}
Let $b:=\sum_{k=1}^\fz b_k$. By \eqref{eq 4.2x}, we have
\begin{eqnarray}\label{eq 4.x0}
&&w\lf(\lf\{x\in\rn:\ \lf|\nabla L_w^{-1/2}(f)(x)\r|>3\az\r\}\r)\\
&&\hs\le w\lf(\lf\{x\in\rn:\ \lf|\nabla L_w^{-1/2}(g)(x)\r|>\az\r\}\r)\noz\\
&&\hs\hs+w\lf(\lf\{x\in\rn:\ \lf|\nabla L_w^{-1/2}(b)(x)\r|>2\az\r\}\r)
=:{\rm I}+{\rm II}.\noz
\end{eqnarray}
We first estimate {\rm I}.
By the fact that $\nabla L_w^{-1/2}$ is bounded on $L^2(w,\,\rn)$ (see \cite[Theorem 1.1]{CR13}),
$p=\frac{2n}{n+1}<2$, \eqref{eq 4.2}, \eqref{eq 4.3}, \eqref{eq 4.4}
and \eqref{eq 4.5}, we see that
\begin{eqnarray}\label{eq 4.x1}
{\rm I}
&&\le \frac{1}{\az^2}\int_\rn \lf|\nabla L_w^{-1/2}(g)(x)\r|^2w(x)\,dx
\ls \frac{1}{\az^2}\int_\rn |g(x)|^2w(x)\,dx\\
&&\ls \frac{1}{\az^p}\int_\rn |g(x)|^pw(x)\,dx\noz\\
&&\ls \frac{1}{\az^p}\lf\{\int_{\rn}|f(x)|^p w(x)\,dx
+ \sum_{k=1}^\fz \int_{B_k}|b_k(x)|^p w(x)\,dx\r\}\noz\\
&&\ls\frac{1}{\az^p}\int_\rn |f(x)|^pw(x)\,dx.\noz
\end{eqnarray}

Next, we prove
\begin{eqnarray}\label{eq 4.6}
{\rm II}
:=w\lf(\lf\{x\in\rn:\ \lf|\nabla L_w^{-1/2}(b)(x)\r|>2\az\r\}\r)
\ls \frac{1}{\az^p}\int_\rn |f(x)|^pw(x)\,dx.
\end{eqnarray}
We claim that, to prove \eqref{eq 4.6}, it suffices to show that
there exists a positive constant $C_{(w,\,p)}$, depending on $w$ and $p$,
such that, for any $\az\in(0,\,\fz)$ and $N\in\nn$,
\begin{eqnarray}\label{eq 4.7}
\qquad w\lf(\lf\{x\in\rn:\ \lf|\nabla L_w^{-1/2}\lf(\sum_{k=1}^N b_k\r)(x)\r|>2\az\r\}\r)
\le \frac{C_{(w,\,p)}}{\az^p}\int_\rn |f(x)|^pw(x)\,dx.
\end{eqnarray}
Indeed, for any $f\in L^2(w,\,\rn)\cap L^p(w,\,\rn)$, by the proof of
\cite[p.\,17, Theorem 2]{S93}, it is easy to see that
$b=\lim_{N\to\fz}\sum_{k=1}^N b_k$ in $L^2(w,\,\rn)$.
Let $S_N:=\sum_{k=1}^N b_k$. By the fact that
$\nabla L_w^{-1/2}$ is bounded on $L^2(w,\,\rn)$, we know that there
exists a subsequence of $\{\nabla L_w^{-1/2}(S_N)\}_{N=1}^\fz$
(without loss of generality, we may use the same notation as the original sequence)
such that, for almost every $x\in\rn$,
$$\lim_{N\to\fz}\nabla L_w^{-1/2}(S_N)(x)=\nabla L_w^{-1/2}(b)(x).$$
By this, Lemma \ref{lem fatou} and \eqref{eq 4.7}, we see that
\begin{eqnarray*}
&&\az^p w\lf(\lf\{x\in\rn:\ \lf|\nabla L_w^{-1/2}(b)(x)\r|>2\az\r\}\r)\\
&&\hs\le \lf\|\lf|\nabla L_w^{-1/2}(b)\r|\r\|_{L^{p,\fz}(w,\,\rn)}^p
=\lf\|\lim_{N\to\fz}\lf|\nabla L_w^{-1/2}(S_N)\r|\r\|_{L^{p,\fz}(w,\,\rn)}^p\\
&&\le \liminf_{N\to\fz}\lf\|\nabla L_w^{-1/2}(S_N)\r\|_{L^{p,\fz}(w,\,\rn)}^p
\ls\int_\rn |f(x)|^pw(x)\,dx,
\end{eqnarray*}
which implies \eqref{eq 4.6}.

Next, we prove \eqref{eq 4.7}. Let $T:=\nabla L_w^{-1/2}$.
Fix some $m\in\nn$ satisfying $m>\frac{n-1}{2}$.
For any $k\in\nn$,
let $T_k:=T(I-e^{-t_k L_w})^m$ and $B_k^\ast:=2B_k$
, where $t_k:=r_k^2$ and $r_k\in (0,\,\fz)$ denotes the
radius of $B_k$.
For any $N\in\nn$ and almost every $x\in\rn$, we write
\begin{eqnarray*}
T(S_N)(x)
=\sum_{k=1}^N T(b_k)(x)=\sum_{k=1}^N T_k(b_k)(x)+ \sum_{k=1}^N (T-T_k)(b_k)(x).
\end{eqnarray*}
Hence,
\begin{eqnarray}\label{eq 4.x2}
&&w\lf(\lf\{x\in\rn:\ \lf|T(S_N)(x)\r|>2\az\r\}\r)\\
&&\hs\le w\lf(\lf\{x\in\rn:\ \lf|\sum_{k=1}^N T_k(b_k)(x)\r|>\az\r\}\r)\noz\\
&&\hs\hs+w\lf(\lf\{x\in\rn:\ \lf|\sum_{k=1}^N (T-T_k)(b_k)(x)\r|>\az\r\}\r)\noz\\
&&\hs\le w\lf(\bigcup_{k=1}^N B_k^\ast\r)
+w\lf(\lf\{x\in\lf(\bigcup_{k=1}^N B_k^\ast\r)^\com:\ \lf|\sum_{k=1}^N T_k(b_k)(x)\r|>\az\r\}\r)\noz\\
&&\hs\hs+w\lf(\lf\{x\in\rn:\ \lf|\sum_{k=1}^N (T-T_k)(b_k)(x)\r|>\az\r\}\r)\noz\\
&&\hs=:{\rm II}_1+{\rm II}_2+{\rm II}_3.\noz
\end{eqnarray}

We first estimate ${\rm II}_1$. By the fact that $w\in A_2(\rn)$, Lemma \ref{lem Ap-2}
and \eqref{eq 4.4}, we see that
\begin{eqnarray}\label{eq 4.x3}
{\rm II}_1
\le \sum_{k=1}^N w(B_k^\ast)\ls \sum_{k=1}^N w(B_k)\ls\frac{1}{\az^p}
\int_\rn |f(x)|^pw(x)\,dx.
\end{eqnarray}

For ${\rm II}_3$, by the Chebyshev inequality and
the fact that $T$ is bounded on $L^2(w,\,\rn)$, we have
\begin{eqnarray*}
{\rm II}_3
&&=w\lf(\lf\{x\in\rn:\ \lf|T\lf(\sum_{k=1}^N\lf[I-(I-e^{-t_k L_w})^m\r](b_k)\r)(x)\r|>\az\r\}\r)\\
&&\ls\frac{1}{\az^2}\lf\|\sum_{k=1}^N\lf(I-(I-e^{-t_k L_w})^m\r)(b_k)\r\|^2_{L^2(w,\,\rn)}.
\end{eqnarray*}
From this and the fact that
\begin{eqnarray*}
I-(I-e^{-t_k L_w})^m=I-\sum_{j=0}^m\binom{m}{j}e^{-jt_kL_w}
=-\sum_{j=1}^m\binom{m}{j}e^{-jt_kL_w},
\end{eqnarray*}
where $\binom{m}{j}$ denotes the binomial coefficients,
it follows that
\begin{eqnarray}\label{eq 4.8}
{\rm II}_3\ls\frac{1}{\az^2}\sum_{j=1}^m
\lf\|\sum_{k=1}^N e^{-jt_kL_w}(b_k)\r\|_{L^2(w,\,\rn)}^2.
\end{eqnarray}
For any $k,\,l\in\nn$, let $S(l,\,k):=2^{l+1}B_k\setminus 2^lB_k$ and
$S(0,\,k):=2B_k$.
For any $h\in L^2(w,\,\rn)$ with $\|h\|_{L^2(w,\,\rn)}=1$,
let $h_{(l,\,k)}:=h\chi_{S(l,\,k)}$.
Then, for any $j\in\{1,\,\ldots,\,m\}$, from \eqref{eq 4.3},
the H\"{o}lder inequality, Lemmas \ref{lem poincare} and \ref{lem 5.1}, we deduce that
\begin{eqnarray*}
&&\lf|\int_\rn \sum_{k=1}^N e^{-jt_kL_w}(b_k)(x)h(x)w(x)\,dx\r|\\
&&\hs=\lf|\sum_{k=1}^N\sum_{l=0}^\fz\int_{B_k}b_k(x) e^{-jt_kL^\ast_w}(h_{(l,\,k)})(x)w(x)\,dx\r|\\
&&\hs=\lf|\sum_{k=1}^N\sum_{l=0}^\fz\int_{B_k}b_k(x)\lf[
e^{-jt_kL^\ast_w}(h_{(l,\,k)})(x)-\lf(e^{-jt_kL^\ast_w}(h_{(l,\,k)})\r)_{B_k}\r]
w(x)\,dx\r|\\
&&\hs\le\sum_{k=1}^N\sum_{l=0}^\fz\|b_k\|_{L^p(w,\,B_k)}
\lf\|e^{-jt_kL^\ast_w}(h_{(l,\,k)})-\lf(e^{-jt_kL^\ast_w}(h_{(l,\,k)})\r)_{B_k}\r\|_{L^{p'}(w,\,B_k)}\\
&&\hs\ls\az\sum_{k=1}^N\sum_{l=0}^\fz[w(B_k)]^{\frac1p}
[w(B_k)]^{\frac1{p'}-\frac12}\lf\|\sqrt{t_k}\nabla e^{-jt_kL^\ast_w}(h_{(l,\,k)})\r\|_{L^2(w,\,\rn)}\\
&&\hs\ls\az\sum_{k=1}^N\sum_{l=0}^\fz[w(B_k)]^{\frac12}
e^{-\frac{[d(S(l,\,k),\,B_k)]^2}{cjt_k}}\|h_{(l,\,k)}\|_{L^2(w,\,S(l,\,k))}\\
&&\hs\ls\az\sum_{k=1}^N\sum_{l=0}^\fz[w(B_k)]^{\frac12} e^{-c4^l}\|h\|_{L^2(w,\,S(l,\,k))},
\end{eqnarray*}
where $(e^{-jt_kL_w^\ast}(h_{(l,k)}))_{B_k}$ is as in \eqref{eq mean}
with $u$ replaced by $e^{-jt_kL_w^\ast}(h_{(l,k)})$ and $B$ replaced by $B_k$.
By this, Lemma \ref{lem Ap-2}, the Kolmogrov lemma
(see, for example, \cite[Lemma 5.16]{Du01})
and the fact that $\|h\|_{L^2(w,\,\rn)}=1$,
we have
\begin{eqnarray*}
&&\lf|\int_\rn \sum_{k=1}^N e^{-jt_kL_w}(b_k)(x)h(x)w(x)\,dx\r|\\
&&\hs\ls\az\sum_{k=1}^N\sum_{l=0}^\fz[w(B_k)]^{\frac12} e^{-c4^l}\|h\|_{L^2(w,\,S(l,\,k))}\\
&&\hs\ls\az\sum_{k=1}^N\sum_{l=0}^\fz[w(B_k)]^{\frac12} e^{-c4^l}
\lf[w\lf(2^{l+1}B_k\r)\r]^{\frac12}\lf[\frac{1}{w\lf(2^{l+1}B_k\r)}
\int_{2^{l+1}B_k}|h(x)|^2w(x)\,dx\r]^{\frac12}\\
&&\hs\ls\az\sum_{k=1}^Nw(B_k)\essinf_{y\in B_k}\lf[M_w(|h|^2)(y)\r]^{\frac12}
\sum_{l=0}^\fz e^{-c4^l}2^{nl}\\
&&\hs\ls\az\sum_{k=1}^N\int_{B_k}\essinf_{y\in B_k}\lf[M_w(|h|^2)(y)\r]^{\frac12}w(x)\,dx\\
&&\hs\ls\az\int_{\cup_{k=1}^\fz B_k}\lf[M_w(|h|^2)(x)\r]^{\frac12}w(x)\,dx\\
&&\hs\ls\az\lf[w\lf(\bigcup_{k=1}^\fz B_k\r)\r]^{\frac12}\lf\||h|^2\r\|^{\frac12}_{L^1(w,\,\rn)}
\ls\az\lf[w\lf(\bigcup_{k=1}^\fz B_k\r)\r]^{\frac12},
\end{eqnarray*}
where $M_w$ is as in \eqref{eq hl}.
This, together with \eqref{eq 4.4} and \eqref{eq 4.5}, implies that,
for any $j\in\{1,\,\ldots,\,m\}$,
\begin{eqnarray*}
\lf\|\sum_{k=1}^N e^{-jt_kL_w}(b_k)\r\|_{L^2(w,\,\rn)}^2
\ls\frac{1}{\az^{p-2}}\int_\rn|f(x)|^pw(x)\,dx.
\end{eqnarray*}
By this and \eqref{eq 4.8}, we know that
\begin{eqnarray}\label{eq 4.0}
{\rm II}_3\ls \frac{1}{\az^{p}}\int_\rn|f(x)|^pw(x)\,dx.
\end{eqnarray}

Next, we estimate ${\rm II}_2$. For any $N\in\nn$,
let $E_N^\ast:=(\cup_{k=1}^N B_k^\ast)^\com$. Then it is easy to see that
\begin{eqnarray}\label{eq 4.9}
{\rm II}_2
\le\frac{1}{\az^2}\lf\|\sum_{k=1}^N T_k(b_k)\r\|^2_{L^2(w,\,E^\ast_N)}.
\end{eqnarray}
For any $k\in\{1,\,\ldots,\,N\}$, let $T_k^\ast$ be the adjoint operator of $T_k$,
namely,
$$T_k^\ast=\lf(\nabla L_w^{-1/2}\lf(I-e^{-tL_w}\r)^m\r)^\ast.$$
For any $\vec{h}:=(h_1,\,\ldots,\,h_n)$
satisfying $\|\vec{h}\|_{L^2(w,\,\rn)}:=\|[\sum_i^n|h_i|^2]^{1/2}\|_{L^2(w,\,\rn)}=1$ with
$\supp h_i\st E_N^\ast$, $i\in\{1,\,\ldots,\,n\}$,
any $l\in\zz_+$ and $k\in\{1,\,\ldots,\,N\}$,
let $\vec{h}_{(l,\,k)}:=\vec{h}\chi_{S(l,\,k)}$.
Since, for any $k\in\{1,\,\ldots,\,m\}$ and $i\in\{1,\,\ldots,\,n\}$,
$\supp h_i\st E_N^\ast$ and $E_N^\ast\st (B_k^\ast)^\com$, we see that
$\vec{h}_{(0,\,k)}=\vec{0}$.
From this, the H\"{o}lder inequality, Lemmas \ref{lem poincare}, \ref{lem 5.2}
and \ref{lem Ap-2}, the fact that $m>\frac{n-1}{2}$,
the Kolmogrov lemma (see, for example, \cite[Lemma 5.16]{Du01})
and the fact that $\|\vec{h}\|_{L^2(w,\,\rn)}=1$,
we deduce that
\begin{eqnarray*}
&&\lf|\int_\rn\lf[\sum_{k=1}^N T_k(b_k)(x)\r]\cdot\vec{h}(x)w(x)\,dx\r|\\
&&\hs=\lf|\sum_{k=1}^N\sum_{l=1}^\fz\int_\rn [T_k(b_k)(x)]\cdot\vec{h}_{(l,\,k)}(x)w(x)\,dx\r|\\
&&\hs=\lf|\sum_{k=1}^N\sum_{l=1}^\fz\int_{B_k}b_k(x)T_k^\ast(\vec{h}_{(l,\,k)})(x)w(x)\,dx\r|\\
&&\hs=\lf|\sum_{k=1}^N\sum_{l=1}^\fz\int_{B_k}b_k(x)\lf[T_k^\ast(\vec{h}_{(l,\,k)})(x)
-\lf(T_k^\ast(\vec{h}_{(l,\,k)})\r)_{B_k}\r]w(x)\,dx\r|\\
&&\hs\le\sum_{k=1}^N\sum_{l=1}^\fz\|b_k\|_{L^p(w,\,B_k)}
\lf\|T_k^\ast\lf(\vec{h}_{(l,\,k)}\r)-\lf(T_k^\ast\lf(\vec{h}_{(l,\,k)}\r)\r)_{B_k}\r\|
_{L^{p'}(w,\,\rn)}\\
&&\hs\ls\az\sum_{k=1}^N\sum_{l=1}^\fz[w(B_k)]^{\frac1p}[w(B_k)]^{\frac1{p'}-\frac12}
\lf\|\sqrt{t_k}\nabla T_k^\ast\lf(\vec{h}_{(l,\,k)}\r)\r\|_{L^2(w,\,B_k)}\\
&&\hs\ls\az\sum_{k=1}^N\sum_{l=1}^\fz[w(B_k)]^{\frac12}
\lf(\frac{[d(B_k,\,S(l,\,k))]^2}{t_k}\r)^{-(m+\frac12)}\|\vec{h}\|_{L^2(w,\,S(l,\,k))}\\
&&\hs\ls\az\sum_{k=1}^N\sum_{l=1}^\fz[w(B_k)]^{\frac12}
2^{-2l(m+\frac12)}\lf[w(2^{l+1}B_k)\r]^{\frac12}\lf[\frac{1}{w(2^{l+1}B_k)}
\int_{2^{l+1}B_k}|\vec{h}(y)|^2w(y)\,dy\r]^{\frac12}\\
&&\hs\ls\az\sum_{k=1}^N w(B_k)\essinf_{y\in B_k}\lf[M_w(|\vec{h}|^2)(y)\r]^{\frac12}
\sum_{l=1}^\fz 2^{-2l(m+\frac12-\frac{n}{2})}\\
&&\hs\ls\az\int_{\cup_{k=1}^\fz B_k}\lf[M_w(|\vec{h}|^2)(x)\r]^{\frac12}w(x)\,dx\\
&&\hs\ls\az\lf[w\lf(\bigcup_{k=1}^\fz B_k\r)\r]^{\frac12}\lf\||\vec{h}|^2\r\|^{\frac12}_{L^1(w,\,\rn)}
\ls\az\lf[w\lf(\bigcup_{k=1}^\fz B_k\r)\r]^{\frac12},
\end{eqnarray*}
where $M_w$ is as in \eqref{eq hl} and
$(T_k^\ast\lf(\vec{h}_{(l,\,k)}\r))_{B_k}$ is as in \eqref{eq mean}
with $u$ and $B$ replaced by $T_k^\ast\lf(\vec{h}_{(l,\,k)}\r)$ and $B_k$, respectively.
This, together with \eqref{eq 4.4} and \eqref{eq 4.5}, implies that
\begin{eqnarray*}
\lf\|\sum_{k=1}^N T_k(b_k)\r\|^2_{L^2(w,\,E^\ast_N)}
\ls \frac{1}{\az^{p-2}}\int_\rn|f(x)|^pw(x)\,dx.
\end{eqnarray*}
Combining this and \eqref{eq 4.9}, we have
$${\rm II}_2\ls\frac{1}{\az^{p}}\int_\rn|f(x)|^pw(x)\,dx.$$
This, together with \eqref{eq 4.0}, \eqref{eq 4.x3} and \eqref{eq 4.x2}, implies \eqref{eq 4.7}.
Hence, \eqref{eq 4.6} holds true. Combining \eqref{eq 4.6} and \eqref{eq 4.x1},
we then complete the proof of Theorem \ref{pro 4}.
\end{proof}

We are now in a position to prove Proposition \ref{thm main1}.
\begin{proof}[Proof of Proposition \ref{thm main1}]
We first prove (i).
Indeed, from Theorem \ref{pro 4}, it follows that, for any $f\in L^p(w,\,\rn)$,
\begin{eqnarray*}
\lf\|\nabla L_w^{-1/2}(f)\r\|_{L^p(w,\,\rn)}\ls \|f\|_{L^p(w,\,\rn)}.
\end{eqnarray*}
This, together with Proposition \ref{thm equi}, implies that,
for any $p\in (\frac{2n}{n+1},\,2]$ and $f\in H_{L_w}^p(\rn)$,
\begin{eqnarray}\label{eq 4.1}
\lf\|\nabla L_{w}^{-1/2}(f)\r\|_{L^p(w,\,\rn)}\ls \|f\|_{H_{L_w}^p(\rn)}.
\end{eqnarray}

Next, we prove (ii).
From \cite[Theorem 1.6]{ZCJY14}, we deduce that, for any $w\in A_q(\rn)$
with $q\in[1,\,1+\frac1n)$ and $f\in H_{L_w}^1(\rn)$,
\begin{eqnarray*}
\lf\|\nabla L_{w}^{-1/2}(f)\r\|_{L^1(w,\,\rn)}\ls \|f\|_{H_{L_w}^1(\rn)}.
\end{eqnarray*}
Combining this, \eqref{eq 4.1}, Lemma \ref{lem inter-Hardy} and the fact that
$$[L^1(w,\,\rn),\,L^{p_0}(w,\,\rn)]_\theta= L^p(w,\,\rn)$$
where $\theta\in(0,\,1)$,
$1/p=(1-\theta)+\theta/{p_0}$ and $p_0\in(1,\,\fz)$ (see, for example, \cite[Theorem 5.5.1]{BL76}),
by the well-known properties of interpolation spaces (see, for example, \cite[Theorem 4.1.2]{BL76}),
we obtain \eqref{eq 4.1}
in case $p\in [1,\,\frac{2n}{n+1}]$.
This finishes the proof of Proposition \ref{thm main1}.
\end{proof}

\section{Proof of Proposition \ref{thm main}}\label{s5}
\hskip\parindent
To prove Proposition \ref{thm main}, we need the following local
weighted Sobolev embedding theorem (see \cite[Theorem (1.2)]{FKS82}).
\begin{lem}[\cite{FKS82}]\label{lem imbedding}
Let $n\geq 2$.
For any given $p\in (1,\fz)$ and $w\in A_p(\rn)$, there exist positive constants
$C$ and $\delta$ such that, for any number $k_0\in [1,\,\frac{n}{n-1}+\delta]$,
any ball $B\equiv B(x_B,r_B)$ of $\rn$
with $x_B\in\rn$ and $r_B\in (0,\fz)$, and any $u\in C^\fz_c(B)$,
\begin{equation*}
\lf[\frac{1}{w(B)}\int_B|u(x)|^{k_0p}w(x)\,dx\r]^{\frac1{k_0p}}\le Cr_B
\lf[\frac{1}{w(B)}\int_B|\nab u(x)|^{p}w(x)\,dx\r]^{\frac1p}.
\end{equation*}
\end{lem}

We are now in a position to prove Proposition \ref{thm main}.
\begin{proof}[Proof of Proposition \ref{thm main}]
Let $p\in(\frac{2n}{n+1},\,\frac{2n}{n-1})$
and $w\in A_p(\rn)\cap A_2(\rn)$.
We first show that, for any given $p\in [2,\,\frac{2n}{n-1})$
and any $h\in L^2(w,\,\rn)\cap H_{L_w,{\rm Riesz}}^p(\rn)$,
\begin{eqnarray}\label{eq 5.x0}
\|h\|_{L^p(w,\,\rn)}\ls \lf\|\nabla L_w^{-1/2}(h)\r\|_{L^p(w,\,\rn)}.
\end{eqnarray}

Indeed, by \cite[Theorem 1.1]{CR13}, the fact that $(L_w^{1/2})^\ast=(L_w^\ast)^{1/2}$,
and an argument similar to that used in the proof of \cite[Lemma 2.2]{Au07},
we find that, for any $u,\,v\in H_0^1(w,\,\rn)$,
\begin{eqnarray*}
\int_\rn L_w^{1/2}(u)(x)\ov{(L_w^\ast)^{1/2}(v)(x)}w(x)\,dx
=\int_{\rn}[A(x)\nabla u(x)]\cdot \overline{\nabla v(x)}\,dx,
\end{eqnarray*}
where $A$ is the complex-valued matrix associated to $L_w$, which satisfies the degenerate elliptic
conditions \eqref{degenerate C1} and \eqref{degenerate C2}.
By this, we see that, for any $f\in H_0^1(w,\,\rn)$ and $g\in L^2(w,\,\rn)$,
\begin{eqnarray*}
\int_\rn L_w^{1/2}(f)(x)\ov{g(x)}w(x)\,dx
&&=\int_\rn L_w^{1/2}(f)(x)\ov{(L_w^\ast)^{1/2}\lf((L_w^\ast)^{-1/2}(g)\r)(x)}w(x)\,dx\\
&&=\int_\rn [A(x)\nabla f(x)]\cdot\ov{\nabla (L_w^\ast)^{-1/2}(g)(x)}\,dx.
\end{eqnarray*}
From this, \eqref{degenerate C1} and the H\"{o}lder inequality, it follows that,
for any given $p\in [2,\,\frac{2n}{n-1})$,
any $f\in H_0^1(w,\,\rn)$ and $g\in L^2(w,\,\rn)\cap L^{p'}(w,\,\rn)$,
\begin{eqnarray}\label{eq 5.x1}
\lf|\int_\rn L_w^{1/2}(f)(x)\ov{g(x)}w(x)\,dx\r|
&&\ls\int_\rn \lf|\nabla f(x)\r|\lf|\nabla (L_w^\ast)^{-1/2}(g)(x)\r|w(x)\,dx\\
&&\ls \|\nabla f\|_{L^p(w,\,\rn)}\lf\|\nabla (L_w^\ast)^{-1/2}(g)\r\|_{L^{p'}(w,\,\rn)}.\noz
\end{eqnarray}
Observing that $p'\in(\frac{2n}{n+1},\,2]$, by Theorem \ref{pro 4}, we see that
$$\lf\|\nabla (L_w^\ast)^{-1/2}(g)\r\|_{L^{p'}(w,\,\rn)}\ls \|g\|_{L^{p'}(w,\,\rn)}.$$
By this and \eqref{eq 5.x1}, we conclude that, for any given
$p\in [2,\,\frac{2n}{n-1})$ and any $f\in H_0^1(w,\,\rn)$,
\begin{eqnarray*}
\lf\|L_w^{1/2}(f)\r\|_{L^p(w,\,\rn)}\ls\|\nabla f\|_{L^p(w,\,\rn)},
\end{eqnarray*}
which further implies \eqref{eq 5.x0}.

Therefore, to complete the proof of Proposition \ref{thm main}, we only need to
prove \eqref{eq  5.x0} in case $p\in (\frac{2n}{n+1},\,2)$.
To this end, we first recall some well-known results.
Let $\mathcal{S}(\rn)$ denote the \emph{space of all Schwartz functions} and
$\mathcal{S}'(\rn)$ the \emph{space of all Schwartz distributions}.
For any $p\in [1,\,\fz)$ and $w\in A_p(\rn)$, the weighted Sobolev space
$\dot{W}^{1,p}(w,\,\rn)$ is defined by
\begin{eqnarray*}
\dot{W}^{1,p}(w,\,\rn):=\lf\{f\in\mathcal{S}'(\rn)/\mathbb{C}:\
\sum_{k=1}^n\|\partial_k f\|_{L^p(w,\,\rn)}<\fz\r\},
\end{eqnarray*}
where, for any $k\in\{1,\,\ldots,\,n\}$, $\partial_k f$ denotes the distributional
derivative of $f$.
From \cite[Theorem 2.8(ii) and Remark 4.5(i)]{Bui82}, it follows that,
for any $p\in[1,\,\fz)$ and $w\in A_p(\rn)$,
$$\dot{F}^{1,w}_{p,2}(\rn)=\dot{W}^{1,p}(w,\,\rn),$$
where $\dot{F}^{1,w}_{p,2}(\rn)$ denotes the homogeneous weighted Triebel sapces
(see \cite[p.\,583]{Bui82} for the definition).
By this and \cite[Theorem 6.2]{Bow08}, we conclude that, for any
$\theta\in(0,\,1)$, $1\le p_0\le p_1<\fz$ and $w\in A_{p_0}(\rn)$,
\begin{eqnarray}\label{eq 5.x3}
\lf[\dot{W}^{1,p_0}(w,\,\rn),\,\dot{W}^{1,p_1}(w,\,\rn)\r]_\theta
&&=\lf[\dot{F}^{1,w}_{p_0,2}(\rn),\,\dot{F}^{1,w}_{p_1,2}(\rn)\r]_\theta\\
&&=\dot{F}^{1,w}_{p,2}(\rn)=\dot{W}^{1,p}(w,\,\rn),\noz
\end{eqnarray}
where $1/p=(1-\theta)/{p_0}+\theta/{p_1}$.

To prove \eqref{eq  5.x0} in case $p\in (\frac{2n}{n+1},\,2)$ and $w\in A_p(\rn)$,
we claim that it suffices to show that,
for any $\az\in (0,\,\fz)$ and any $f\in \dot{W}^{1,p}(w,\,\rn)$,
\begin{eqnarray}\label{eq 5.xx}
w\lf(\lf\{x\in\rn:\ \lf|S_1\lf(\sqrt{L_w}(f)\r)(x)\r|>\az\r\}\r)
\ls\frac{1}{\az^p}\int_\rn |\nabla f(x)|^pw(x)\,dx,
\end{eqnarray}
where, for any $h\in L^2(w,\,\rn)$ and $x\in\rn$,
$$S_1(h)(x):=\lf[\iint_{\Gamma(x)}\lf|t\sqrt{L_w}e^{-t^2L_w}(h)(y)\r|^2w(y)\,
\frac{dy}{w(B(x,t))}\,\frac{dt}{t}\r]^{\frac12}.$$
Indeed, since $L_w$ has a bounded $H_\fz$ functional calculus in $L^2(w,\,\rn)$,
we know that
$S_1$ is bounded on $L^2(w,\,\rn)$ (see, for example, \cite[p.\,487]{BL11} or \cite{ADM96}).
Thus, by this and \cite[Theorem 1.1]{CR13},
we know that, for any $f\in\mathcal{S}(\rn)\st H_0^1(w,\,\rn)$,
\begin{eqnarray}\label{eq 5.x2}
\lf\|S_1\lf(\sqrt{L_w}(f)\r)\r\|_{L^2(w,\,\rn)}\ls\lf\|\sqrt{L_w}(f)\r\|_{L^2(w,\,\rn)}
\sim \|\nabla f\|_{L^2(w,\,\rn)}.
\end{eqnarray}
Since $\mathcal{S}(\rn)\cap \dot{F}^{1,w}_{2,2}(\rn)$ is dense in $\dot{F}^{1,w}_{2,2}(\rn)$
(see \cite[p.\,153]{Bow08}), we know that
$\mathcal{S}(\rn)$ is dense in $\dot{W}^{1,2}(w,\,\rn)$.
From this and a limiting procedure, we deduce that, for all $f\in \dot{W}^{1,2}(w,\,\rn)$,
\eqref{eq 5.x2} holds true.
By \cite[Theorem 5.3.1]{BL76}, we find that,
for any $p_0\in[1,\,\fz)$ and $\theta\in(0,\,1)$,
\begin{eqnarray*}
\lf[L^{p_0,\,\fz}(w,\,\rn),\,L^{2,\,\fz}(w,\,\rn)\r]_{\theta}=L^p(w,\,\rn),
\end{eqnarray*}
where $1/p=(1-\theta)/{p_0}+\theta/2$.
Combining this, \eqref{eq 5.x2}, \eqref{eq 5.x3} and \eqref{eq 5.xx},
by the well-known properties of interpolation spaces (see, for example, \cite[Theorem 4.1.2]{BL76}),
we see that, for any $q\in (p,\,2)$ with $p\in (\frac{2n}{n+1},\,2)$
and $f\in L^q(w,\,\rn)$,
\begin{eqnarray*}
\lf\|S_1\lf(\sqrt{L_w}(f)\r)\r\|_{L^q(w,\,\rn)}\ls \|\nabla f\|_{L^q(w,\,\rn)}.
\end{eqnarray*}
This, together with Proposition \ref{pro 1}, implies that, for all
$q\in(\frac{2n}{n+1},\,2)$ and $h\in L^2(w,\,\rn)\cap H_{L_w,{\rm Riesz}}^q(\rn)$,
\begin{eqnarray*}
\|h\|_{H_{L_w}^q(\rn)}\sim \lf\|S_1(h)\r\|_{L^q(w,\,\rn)}\ls
\lf\|\nabla L_w^{-1/2}(h)\r\|_{L^q(w,\,\rn)}.
\end{eqnarray*}

Next, we prove that, for any $f\in\mathcal{S}(\rn)$, \eqref{eq 5.xx} holds true.
Then, by Lemma \ref{lem fatou} and a density argument, we further know that, for
any $f\in \dot{W}^{1,p}(\rn)$, \eqref{eq 5.xx} holds true.
For any $f\in\mathcal{S}(\rn)$,
by the Calder\'{o}n-Zygmund decomposition of weighted Sobolev
spaces (see, for example, \cite[Proposition 1.1]{AC05} or \cite[Lemma 6.6]{AM06iii}),
we conclude that
there exist positive constants $C$ and $N$ such that, for any $\az\in(0,\,\fz)$,
there exist a collection $\{B_i\}_{i=1}^\fz$ of balls of $\rn$,
a family of functions, $\{b_i\}_{i=1}^\fz\st C^1(\rn)$,
and an almost everywhere Lipschitz function $g$ such that the following properties hold true:
\begin{eqnarray}\label{eq 5.y1}
f(x)=g(x)+\sum_{i=1}^\fz b_i(x)\ \ \text{for almost every}\ x\in\rn;
\end{eqnarray}
\begin{eqnarray}\label{eq 5.y2}
\|\nabla g\|_{L^p(w,\,\rn)}\le C\|\nabla f\|_{L^p(w,\,\rn)},\ \
|\nabla g(x)|\le C\az\ \ \text{for almost every}\ x\in\rn;
\end{eqnarray}
\begin{eqnarray}\label{eq 5.y3}
\supp b_i\st B_i\ \ \text{and}\ \int_{B_i}|\nabla b_i(x)|^pw(x)\,dx\le C\az^p w(B_i);
\end{eqnarray}
\begin{eqnarray}\label{eq 5.y4}
\sum_{i=1}^\fz w(B_i)\le C\az^{-p}\int_\rn|\nabla f(x)|^pw(x)\,dx;
\end{eqnarray}
\begin{eqnarray}\label{eq 5.y5}
\sum_{i=1}^\fz \chi_{B_i}(x)\le N,
\end{eqnarray}
here and hereafter, for any $k\in\nn$, $C^k(\rn)$ denotes the space of all functions
possessing continuous derivatives up to order $k$ on $\rn$.
Moreover, by the proof of \cite[Proposition 1.1]{AC05},
we further see that, for any $i\in\nn$,
\begin{eqnarray}\label{eq a5}
b_i=(f-f_{B_i})\zeta_i,
\end{eqnarray}
where $0\le\zeta_i\le1$, $\zeta_i\in C^1(\rn)$ with $\supp\zeta_i\st B_i$, and
$f_{B_i}$ is as in \eqref{eq mean} with $u$ and $B$ replaced by $f$ and $B_i$, respectively.
By this, \eqref{eq 5.y5} and the fact that $f\in\mathcal{S}(\rn)\st L^2(w,\,\rn)$,
it is easy to see that $\sum_{i=1}^\fz b_i\in L^2(w,\,\rn)$.
From this and the fact that $L_w e^{-t^2L_w}$ is bounded on $L^2(w,\,\rn)$ for any $t\in(0,\,\fz)$,
it follows that
\begin{eqnarray}\label{eq a0}
\lim_{N\to\fz}\sum_{i=1}^N tL_w e^{-t^2L_w}(b_i)
=tL_w e^{-t^2L_w}\lf(\sum_{i=1}^\fz b_i\r)\ \ \text{in}\ L^2(w,\,\rn).
\end{eqnarray}
For any $N\in\nn$, let $S_N:=\sum_{i=1}^N tL_w e^{-t^2L_w}(b_i)$.
By \eqref{eq a0}, we further know that there exists a subsequence of $\{S_N\}_{N\in\nn}$
(without loss of generality, we may use the same notation as the original sequence)
such that, for almost every $x\in\rn$,
\begin{eqnarray}\label{eq a4}
tL_we^{-t^2L_w}\lf(\sum_{i=1}^\fz b_i\r)(x)=
\lim_{N\to\fz}S_N(x)=\sum_{i=1}^\fz tL_we^{-t^2L_w}(b_i)(x).
\end{eqnarray}
Moreover, from \eqref{eq a5} and Lemma \ref{lem poincare}, we deduce that, for any $i\in\nn$,
\begin{eqnarray*}
\int_{B_i}\frac{|b_i(x)|^2}{r^2_{B_i}}w(x)\,dx
\ls\int_{B_i}\frac{|f(x)-f_{B_i}|^2}{r^2_{B_i}}w(x)\,dx
\ls \int_{B_i}|\nabla f(x)|^2w(x)\,dx,
\end{eqnarray*}
which, together with \eqref{eq 5.y5}, implies that
$\sum_{i=1}^\fz \frac{b_i}{r_{B_i}}\in L^2(w,\,\rn)$.
By an argument similar to that used in the proof of \eqref{eq a4},
we find that, for any $t\in(0,\,\fz)$ and almost every $x\in\rn$,
\begin{eqnarray*}
tL_we^{-t^2L_w}\lf(\sum_{i=1}^\fz \frac{b_i}{r_{B_i}}\r)(x)=
\sum_{i=1}^\fz tL_we^{-t^2L_w}\lf(\frac{b_i}{r_{B_i}}\r)(x).
\end{eqnarray*}
By this, \eqref{eq a4} and the Minkowski inequality, for any $x\in\rn$, we find that
\begin{eqnarray*}
&&S_1\lf(\sqrt{L_w}(f)\r)(x)\\
&&\hs\le S_1\lf(\sqrt{L_w}(g)\r)(x)
+\lf[\iint_{\bgz(x)}\lf|
\sum_{i=1}^\fz tL_we^{-t^2L_w}(b_i)(y)\chi_{(0,\,r_{B_i})}(t)\r|^2
\frac{w(y)\,dy}{w(B(x,t))}\,\frac{dt}{t}\r]^{\frac12}\\
&&\hs\hs+\lf[\iint_{\bgz(x)}\lf|
\sum_{i=1}^\fz tL_we^{-t^2L_w}(b_i)(y)\chi_{[r_{B_i},\,\fz)}(t)\r|^2
\frac{w(y)\,dy}{w(B(x,t))}\,\frac{dt}{t}\r]^{\frac12}\\
&&\hs\le S_1\lf(\sqrt{L_w}(g)\r)(x)
+\sum_{i=1}^\fz\lf[\int_0^{r_{B_i}}\int_{B(x,t)}\lf|tL_we^{-t^2L_w}(b_i)(y)\r|^2
\frac{w(y)\,dy}{w(B(x,t))}\,\frac{dt}{t}\r]^{\frac12}\\
&&\hs\hs+\lf[\int_0^\fz\int_{B(x,t)}\lf|t^2L_we^{-t^2L_w}
\lf(\sum_{i=1}^\fz\frac{b_i}{r_{B_i}}\r)(y)\r|^2
\frac{w(y)\,dy}{w(B(x,t))}\,\frac{dt}{t}\r]^{\frac12}\\
&&\hs=:{\rm I}_1(x)+\sum_{i=1}^\fz {\rm I}_{2,i}(x)+{\rm I}_3(x).
\end{eqnarray*}
By this, we know that
\begin{eqnarray}\label{eq aa}
&&w\lf(\lf\{x\in\rn:\ \lf|S_1\lf(\sqrt{L_w}(f)\r)(x)\r|>\az\r\}\r)\\
&&\hs\le w\lf(\lf\{x\in\rn:\ {\rm I}_1(x)>\frac{\az}{3}\r\}\r)
+w\lf(\lf\{x\in\rn:\ \sum_{i=1}^\fz {\rm I}_{2,i}(x)>\frac{\az}{3}\r\}\r)\noz\\
&&\hs\hs+w\lf(\lf\{x\in\rn:\ {\rm I}_3(x)>\frac{\az}{3}\r\}\r)\noz\\
&&\hs=:{\rm A}_1+{\rm A}_2+{\rm A}_3.\noz
\end{eqnarray}

We first estimate ${\rm A}_1$. Using the Chebyshev inequality, \eqref{eq 5.x2}
and \eqref{eq 5.y2}, we have
\begin{eqnarray}\label{eq a1}
{\rm A}_1&&\ls \frac{1}{\az^2}\int_\rn \lf|S_1\lf(\sqrt{L_w}(g)\r)(x)\r|^2w(x)\,dx
\ls\frac{1}{\az^2}\int_\rn |\nabla g(x)|^2w(x)\,dx\\
&&\ls\frac{1}{\az^2}\int_\rn \az^{2-p}|\nabla g(x)|^pw(x)\,dx
\ls\frac{1}{\az^p}\int_\rn|\nabla f(x)|^pw(x)\,dx.\noz
\end{eqnarray}

Next, we estimate ${\rm A_3}$. By the Chebyshev inequality,
\eqref{eq Spb} and \eqref{eq 5.y5}, we conclude that
\begin{eqnarray}\label{eq 5.x4}
{\rm A}_3&&\ls\frac{1}{\az^p}
\int_\rn\lf|S_{L_w}\lf(\sum_{i=1}^\fz \frac{b_i}{r_{B_i}}\r)(x)\r|^pw(x)\,dx
\ls\frac{1}{\az^p}\int_\rn\lf|\sum_{i=1}^\fz\frac{b_i(x)}{r_{B_i}}\r|^pw(x)\,dx\\
&&\ls\frac{1}{\az^p}\sum_{i=1}^\fz\int_{B_i}\frac{|b_i(x)|^p}{r^p_{B_i}}w(x)\,dx\noz
\end{eqnarray}
Observing that $b_i\in C^1(\rn)$ with $\supp b_i\st B_i$,
by Lemma \ref{lem imbedding}, \eqref{eq 5.y3} and \eqref{eq 5.y4}, we see that
\begin{eqnarray}\label{eq 5.x6}
\lf[\int_{B_i}|b_i(x)|^pw(x)\,dx\r]^{\frac{1}{p}}\ls r_{B_i}
\lf[\int_{B_i}|\nabla b_i(x)|^pw(x)\,dx\r]^{\frac1p}
\ls r_{B_i}\az\lf[w(B_i)\r]^{\frac1p}.
\end{eqnarray}
This, together with \eqref{eq 5.x4}, implies that
\begin{eqnarray}\label{eq a3}
{\rm A}_3\ls \sum_{i=1}^\fz w(B_i)\ls\frac{1}{\az^p}\int_\rn|\nabla f(x)|^pw(x)\,dx.
\end{eqnarray}

Finally, we estimate ${\rm A_2}$. From \eqref{eq aa}, Lemma \ref{lem Ap-2},
the Chebyshev inequality and \eqref{eq 5.y4},
it follows that
\begin{eqnarray}\label{eq 5.x8}
{\rm A}_2
&&=w\lf(\lf\{x\in\rn:\ \sum_{i=1}^\fz {\rm I}_{2,i}(x)>\frac{\az}{3}\r\}\r)\\
&&\le\sum_{i=1}^\fz w(4B_i)+w\lf(\lf\{x\in\lf(\bigcup_{i=1}^\fz 4B_i\r)^\com:\
\sum_{i=1}^\fz {\rm I}_{2,i}(x)>\frac{\az}{3}\r\}\r)\noz\\
&&\ls\frac{1}{\az^p}\int_\rn|\nabla f(x)|^pw(x)\,dx
+\frac{1}{\az^2}\int_\rn\lf|\sum_{i=1}^\fz {\rm I}_{2,i}(x)
\chi_{(4B_i)^\com}(x)\r|^2w(x)\,dx\noz\\
&&\sim\frac{1}{\az^p}\int_\rn|\nabla f(x)|^pw(x)\,dx\noz\\
&&\hs+\frac{1}{\az^2}\lf\{\sup_{\|u\|_{L^2(w,\,\rn)}=1}
\lf|\int_\rn\lf[\sum_{i=1}^\fz{\rm I}_{2,i}(x)\chi_{(4B_i)^\com}(x)\r]
u(x)w(x)\,dx\r|\r\}^2. \noz
\end{eqnarray}
For any $u\in L^2(w,\,\rn)$ with $\|u\|_{L^2(w,\,\rn)}=1$,
by the H\"{o}lder inequality and the Fubini theorem,
we find that
\begin{eqnarray}\label{eq 5.x5}
\qquad&&\lf|\int_\rn\lf[\sum_{i=1}^\fz {\rm I}_{2,i}(x)\chi_{(4B_i)^\com}(x)\r]u(x)w(x)\,dx\r|\\
&&\hs\le\sum_{i=1}^\fz\int_{(4B_i)^\com}|{\rm I}_{2,i}(x)||u(x)|w(x)\,dx\noz\\
&&\hs=\sum_{i=1}^\fz\sum_{j=3}^\fz\int_{U_j(B_i)}|{\rm I}_{2,i}(x)||u(x)|w(x)\,dx\noz\\
&&\hs\le\sum_{i=1}^\fz\sum_{j=3}^\fz
\lf[\int_{U_j(B_i)}\int_0^{r_{B_i}}\int_{B(x,t)}\lf|t^2L_we^{-t^2L_w}(b_i)(y)\r|^2
\frac{w(y)\,dy}{w(B(x,t))}\,\frac{dt}{t^3}w(x)\,dx\r]^{\frac12}\noz\\
&&\hs\hs\times\|u\|_{L^2(w,\,U_j(B_i))}\noz\\
&&\hs\ls\sum_{i=1}^\fz\sum_{j=3}^\fz\lf[\iint_{(y,t)\in R(U_j(B_i)),\,t\in(0,r_{B_i})}
\lf|t^2L_we^{-t^2L_w}(b_i)(y)\r|^2w(y)\,dy\,\frac{dt}{t^3}\r]^{\frac12}\noz\\
&&\hs\hs\times\|u\|_{L^2(w,\,U_j(B_i))}\noz\\
&&\hs\ls\sum_{i=1}^\fz\sum_{j=3}^\fz\lf[\int_0^{r_{B_i}}\int_{2^{j+1}B_i\setminus 2^{j-2}B_i}
\lf|t^2L_we^{-t^2L_w}(b_i)(y)\r|^2w(y)\,dy\,\frac{dt}{t^3}\r]^{\frac12}
\|u\|_{L^2(w,\,U_j(B_i))},\noz
\end{eqnarray}
where $U_j(B_i)$ is as in \eqref{eq-def of ujb} with $B$ replaced by $B_i$ and
$R(U_j(B_i))$ is as in \eqref{eq tent} with $F$ replaced by $U_j(B_i)$.
From Proposition \ref{pro ODEB}, Lemma \ref{lem Ap-2} and \eqref{eq 5.x6},
it follows that there exist positive constants $c$, $\wz c$, $\tz_1$, $\tz_2$ and $\tz$ such that,
for all $t\in(0,\,r_{B_i})$,
\begin{eqnarray*}
&&\lf[\int_{2^{j+1}B_i\setminus 2^{j-2}B_i}|t^2L_we^{-t^2L_w}(b_i)(y)|^2w(y)\,dy\r]^{\frac12}\\
&&\hs\ls 2^{j\tz_1}\lf[\Upsilon\lf(\frac{2^jr_{B_i}}{t}\r)\r]^{\tz_2}
e^{-c\lf(\frac{2^jr_{B_i}}{t}\r)^2}\lf[w(2^jB_i)\r]^{\frac12}
\lf[\frac{1}{w(B_i)}\int_{B_i}|b_i(y)|^pw(y)\,dy\r]^{\frac1p}\\
&&\hs\ls2^{j\tz}e^{-\wz c\lf(\frac{2^jr_{B_i}}{t}\r)^2}[w(B_i)]^{\frac12-\frac1p}r_{B_i}
\az[w(B_i)]^{\frac1p}
\ls2^{j\tz}r_{B_i}\az e^{-\wz c\frac{(2^jr_{B_i})^2}{t^2}}[w(B_i)]^{\frac12}.
\end{eqnarray*}
From this and \eqref{eq 5.x5}, via choosing a positive constant $N\in(\max\{2\theta,\,3\},\,\fz)$,
we deduce that, for any $u\in L^2(w,\,\rn)$ with $\|u\|_{L^2(w,\,\rn)}=1$,
\begin{eqnarray}\label{eq 5.x7}
&&\lf|\int_\rn\lf[\sum_{i=1}^\fz {\rm I}_{2,i}(x)\chi_{(4B_i)^\com}(x)\r]u(x)w(x)\,dx\r|\\
&&\hs\ls\sum_{i=1}^\fz\sum_{j=3}^\fz\az[w(B_i)]^{\frac12}
\lf[\int_0^{r_{B_i}}2^{j(2\tz-N)}e^{-2\wz c\lf(\frac{2^jr_{B_i}}{t}\r)^2}
\lf(\frac{2^jr_{B_i}}{t}\r)^N r^{2-N}_{B_i}t^{N-3}\,dt\r]^{\frac12}\noz\\
&&\hs\hs\times\|u\|_{L^2(w,\,U_j(B_i))}\noz\\
&&\hs\ls\sum_{i=1}^\fz\sum_{j=3}^\fz2^{-j(\frac{N}{2}-\tz)}\az[w(B_i)]^{\frac12}
\|u\|_{L^2(w,\,U_j(B_i))}\noz\\
&&\hs\ls\sum_{i=1}^\fz\sum_{j=3}^\fz2^{-j(\frac{N}{2}-\theta)}\az w(B_i)
\lf[\frac{1}{w(2^jB_i)}\int_{2^jB_i}|u(y)|^2w(y)\,dy\r]^{\frac12}\noz\\
&&\hs\ls\sum_{i=1}^\fz\sum_{j=3}^\fz2^{-j(\frac{N}{2}-\theta)}\az w(B_i)
\inf_{z\in B_i}\lf[M_w(|u|^2)(z)\r]^{\frac12}\noz\\
&&\hs\ls\sum_{i=1}^\fz\az \int_{B_i}\lf[M_w(|u|^2)(z)\r]^{\frac12}w(z)\,dz
\ls\az\int_{\cup_{i=1}^\fz B_i}\lf[M_w(|u|^2)(z)\r]^{\frac12}w(z)\,dz,\noz
\end{eqnarray}
where the Hardy-Littlewood maximal function $M_w$ is as in \eqref{eq hl}.
Using the Kolmogrov lemma (see, for example, \cite[Lemma 5.16]{Du01}), we obtain
\begin{eqnarray*}
\int_{\cup_{i=1}^\fz B_i}\lf[M_w(|u|^2)(z)\r]^{\frac12}w(z)\,dz
\ls \lf[w\lf(\bigcup_{i=1}^\fz B_i\r)\r]^{1-\frac12}\lf\||u|^2\r\|^{\frac12}_{L^1(w,\,\rn)}.
\end{eqnarray*}
This, together with \eqref{eq 5.x7}, \eqref{eq 5.x8}, \eqref{eq 5.y4}, \eqref{eq 5.y5}
and the fact that $\|u\|_{L^2(w,\,\rn)}=1$, implies that
\begin{eqnarray}\label{eq a2}
{\rm A}_2\ls \frac{1}{\az^p}\int_\rn|\nabla f(x)|^pw(x)\,dx
+w\lf(\bigcup_{i=1}^\fz B_i\r)
\ls\frac{1}{\az^p}\int_\rn|\nabla f(x)|^pw(x)\,dx.
\end{eqnarray}
Combining \eqref{eq aa}, \eqref{eq a1}, \eqref{eq a3}
and \eqref{eq a2}, we see that, for any $f\in\mathcal{S}(\rn)$,
\begin{eqnarray*}
w\lf(\lf\{x\in\rn:\ \lf|S_1\lf(\sqrt{L_w}(f)\r)(x)\r|>\az\r\}\r)
\ls\frac{1}{\az^p}\int_\rn|\nabla f(x)|^pw(x)\,dx,
\end{eqnarray*}
which further implies \eqref{eq 5.xx}.
This finishes the proof of Proposition \ref{thm main}.
\end{proof}

\section{Proof of Theorem \ref{cor main}}\label{s6}
\hskip\parindent
In this section, we show that Propositions \ref{thm equi}, \ref{thm main1} and \ref{thm main}
imply Theorem \ref{cor main}.
\begin{proof}[Proof of Theorem \ref{cor main}]
Let $p\in(\frac{2n}{n+1},\,2]$ and $w\in A_p(\rn)$.
For any $f\in C_c^\fz(\rn)$, by \cite[Theorem 1.1]{CR13}, we know
that $L_w^{1/2}(f)\in L^2(w,\,\rn)$. Moreover, it is easy to see that,
for any $p\in (\frac{2n}{n+1},\,2]$ and any $f\in C_c^\fz(\rn)$,
$$\lf\|\nabla L_w^{-1/2}\lf(L_w^{1/2}(f)\r)\r\|_{L^p(w,\,\rn)}
=\|\nabla f\|_{L^p(w,\,\rn)}<\fz,$$
which implies that $L_w^{1/2}(f)\in H^p_{L_w,\,{\rm Riesz}}(\rn)$.
Thus, by Theorems \ref{thm main} and \ref{thm equi},
we conclude that, for $p\in (\frac{2n}{n+1},\,2]$ and any $f\in C^\fz_c(\rn)$,
\begin{eqnarray}\label{eq 6.0}
\|L_w^{1/2}(f)\|_{L^p(w,\,\rn)}\ls\lf\|\nabla L_w^{-1/2}\lf(L_w^{1/2}(f)\r)\r\|_{L^p(w,\,\rn)}
\sim\|\nabla f\|_{L^p(w,\,\rn)},
\end{eqnarray}
which further implies that $L_w^{1/2}(f)\in L^p(w,\,\rn)$.
From this, Theorems \ref{thm equi} and \ref{thm main1}, we deduce that,
for $p\in (\frac{2n}{n+1},\,2]$ and any $f\in C^\fz_c(\rn)$,
\begin{eqnarray*}
\|\nabla f\|_{L^p(w,\,\rn)}=\lf\|\nabla L_w^{-1/2}\lf(L_w^{1/2}(f)\r)\r\|_{L^p(w,\,\rn)}
\ls\lf\|L_w^{1/2}(f)\r\|_{L^p(w,\,\rn)}.
\end{eqnarray*}
This, together with \eqref{eq 6.0}, finishes the proof of Theorem \ref{cor main}.
\end{proof}

\smallskip

\noindent\textbf{Acknowledgements}. The authors would like to thank Professors Jun Cao,
Sibei Yang, Wen Yuan and Renjin Jiang for some helpful conversations on this topic.


\bigskip

\noindent{Dachun  Yang  and Junqiang Zhang (Corresponding author)}

\medskip

\noindent{\small School of Mathematical Sciences, Beijing Normal University,
Laboratory of Mathematics and Complex Systems, Ministry of Education,
Beijing 100875, People's Republic of China}

\smallskip

\noindent{\it E-mails}: \texttt{dcyang@bnu.edu.cn} (D. Yang)

\hspace{1.12cm}\texttt{zhangjunqiang@mail.bnu.edu.cn} (J. Zhang)

\end{document}